\documentclass[11pt, reqno]{amsart}

\usepackage{amsfonts}
\usepackage{amssymb}
\usepackage{xcolor}
\usepackage{amsthm}
\usepackage{amsmath}
\usepackage{enumerate}

\usepackage[style = nature, sorting = nyt]{biblatex}
\usepackage{graphicx}

\newtheorem{thm}{Theorem}[section]
\newtheorem{cor}[thm]{Corollary}

\newtheorem{lemma}[thm]{Lemma}

\newtheorem{conjecture}{Conjecture}
\newtheorem{rmk}[thm]{Remark}

\newtheorem{defn}[thm]{Definition}
\newtheorem{prop}[thm]{Proposition}
\newtheorem{conj}[thm]{Conjecture}
\newtheorem{question}[thm]{Question}

\numberwithin{equation}{section}
\newcommand{\set}[1]{\left\{#1\right\}}
\begin{document}
\title{On the Willmore energy of Mobius bands}
 \author{Jacob Bernstein}\address{Johns Hopkins University \\ 3400 N. Charles St, Baltimore MD 21218 USA}
	\email{jberns15@jhu.edu}
\author{Daniel Ketover}\address{Rutgers University\\  Busch Campus - Hill Center \\ 110 Freylinghausen Road, Piscataway NJ 08854 USA}

 \email{dk927@math.rutgers.edu}

\begin{abstract}
We show that among M\"obius bands in $\mathbb{S}^3$ bounded by a great circle, the minimal Willmore energy is realized by an embedded minimal M\"obius band with Morse index two.  To prove this, we introduce a $2$-parameter ``canonical family" associated to any non-orientable surface in $\mathbb{S}^3$ with boundary a great circle and apply a min-max argument.  The canonical family detects the Euler number of the surface and is inspired by the $5$-parameter family discovered by F.C. Marques and A. Neves detecting the genus of an orientable surface in $\mathbb{S}^3$.   For $\mathbb{Z}_2$-invariant Klein bottles, this reduces R. Kusner's 1989 conjecture that $\tau_{1,2}$ minimizes the Willmore energy for a Klein bottle immersed in $\mathbb{S}^3$ to any of several conjectural characterizations of the Lawson M\"obius band. 
\end{abstract}

\maketitle

\section{Introduction}

Let $C$ be a great circle in $\mathbb{S}^3$.  In this paper, we study the Willmore energy and conformal area of surfaces immersed in $\mathbb{S}^3$ spanning $C$ and their relationship to minimal surfaces spanning $C$.  Among all such surfaces, the simplest are the hemispheres bounded by $C$ of area $2\pi$ which are minimal and minimize both quantities. By a result of R. Hardt and L. Simon \cite{HS}, if $\Sigma$ is an orientable and embedded minimal surface spanning $C$, then $\Sigma$ is a hemisphere.

 In the class of non-orientable surfaces the picture is richer. In 1970, B. Lawson \cite{Lawson} found, for each $k\geq 1$, an immersed minimal M\"obius band spanning $C$ that we denote $\overline{\tau}_{1,2k}$. These surfaces comprise half of his Klein bottles $\tau_{1,2k}$.  The Lawson bands are given explicitly by an algebraic equation (realizing $\mathbb{S}^3\subset\mathbb{R}^4\cong \mathbb{C}^2$):
\begin{equation}
\bar{\tau}_{1,2k}=\{(z,w)\in\mathbb{S}^3\;|\; \mbox{Im}(z^{2k}\overline{w})=0,\;\mbox{Re}(z^{2k}\bar{w})\geq0\}.
\end{equation}
Among these surfaces, only $\overline{\tau}_{1,2}$ is embedded (while the others have $k$ sheets meeting at equal angles at the polar circle to $C$).  In \cite{BernKetExistence}, the authors found infinitely many minimal embeddings $\tilde{\xi}_{k,m}$\footnote{The surfaces are ``twisted" analogs of the \emph{closed} embedded Lawson surfaces  $\xi_{k,m}$ discovered in 1970 (B. Lawson, \cite{Lawson}) which is the reason for the notation.} with boundary $C$, realizing all non-orientable genera except those of the form $2^n+1$ for some $n\in\mathbb{N}$ as well as all possible Euler numbers. This answered a question of R. Hardt and H. Rosenberg \cite{HR}.

While the hemispheres are the simplest minimal surfaces spanning $C$, we expect the Lawson band, $\bar{\tau}_{1,2}$ to be the second simplest:
\begin{conjecture}\label{secondArea}
Let $C$ be a great circle in $\mathbb{S}^3$.  The minimal surface of second lowest area bounded by $C$ is the Lawson M\"obius band $\overline{\tau}_{1,2}$ with 
\begin{equation}
\mbox{Area}(\overline{\tau}_{1,2})= 4\pi E(\frac{\sqrt{3}}{4})\approx 4.84\pi,
\end{equation}
where $E(a)$ is the second elliptic integral\footnote{With the convention $E(a):=\int_0^1\sqrt{1-a^2\sin(\theta)}d\theta$.}.  
\end{conjecture}
This is analogous to the fact,  shown in \cite{MN}, that  the Clifford torus has second lowest area after the great sphere. 
Indeed, the Lawson band $\bar{\tau}_{1,2}$ satisfies a cubic equation while the Clifford torus, which is $\tau_{1,1}$, satisfies a quadratic equation.

In this paper, we introduce a new canonical two-parameter family associated to any non-orientable surface with boundary $C$.  Using a min-max argument applied to the canonical family, we first obtain a characeterization of the least area minimal M\"obius band spanning $C$:
\begin{thm}\label{mainintro}
Any minimal M\"obius band immersed in $\mathbb{S}^3$ with boundary $C$ of lowest area is embedded, has Morse index $2$, and Euler number $\pm 2$.
\end{thm}
Here the \emph{Euler number} of an immersion  $u:\Sigma\to\mathbb{S}^3$ spanning $C$ is the linking number of the ``boundary push-off" curve $u(\partial(T_\varepsilon(\partial\Sigma))$ with $C$, where $T_\varepsilon(A)$ denotes the $\varepsilon$-tubular neighborhood of $A$ in $\Sigma$ and $\varepsilon$ is small enough.  It encodes important topological information about the surfaces spanning $C$. When $\Sigma$ is orientable, the Euler number is zero while the Lawson band, $\bar{\tau}_{1,2}$ has Euler number $\pm 2$ (Section 2, \cite{BernKetExistence}).

The Lawson band, $\bar{\tau}_{1,2}$ has Morse index $2$ (Appendix B, \cite{BernKetProperties}) and we expect, by analogy with F. Urbano's theorem \cite{Urbano},  this to characterize $\bar{\tau}_{1,2}$:
\begin{conjecture}\label{index2}
A minimal surface in $\mathbb{S}^3$ with boundary $C$ and with Morse index at most $2$ is either a hemisphere or (up to ambient isometry) the Lawson band $\overline{\tau}_{1,2}$.
\end{conjecture}
S. Brendle's \cite{BrendleS} resolution of the Lawson conjecture also suggests:
\begin{conjecture}\label{L2Conj}
An embedded minimal M\"{o}bius band in $\mathbb{S}^3$ spanning $C$ and is (up to ambient isometry) the Lawson band $\overline{\tau}_{1,2}$.
\end{conjecture}
See the list of problems in \cite{BernKetExistence} for related questions.

The validity of either conjecture would show the Lawson band is the minimizer in Theorem \ref{mainintro} and provide evidence for Conjecture \ref{secondArea}.  In fact, we only need a weaker form of Conjecture \ref{index2}: the only index $2$ minimal \emph{M\"obius band} is the Lawson band.  
In \cite{BernKetProperties} the authors prove several results toward Conjecture \ref{index2}.  In particular, we show that all the known higher genus embedded examples $\tilde{\xi}_{k,m}$ have Morse index greater than $2$, that there is no counterexample of genus $2$ (i.e. a punctured Klein bottle) and that any M\"obius band counterexample is in a finite dimensional family.

We also expect the Lawson band to be the simplest M\"{o}bius band, or, more generally, non-orientable, surface spanning $C$. Specifically, for a surface $\Sigma$ immersed in $\mathbb{S}^3$ (possibly with boundary), the Willmore or ``bending" energy is given by (\cite{Pozzetta}, \cite{Schatzle}):
\begin{equation}
\mathcal{W}(\Sigma)=\int_\Sigma (1+\frac{1}{4}|\mathbf{H}_\Sigma|^2)\mbox{dA}+\int_{\partial\Sigma}k_g dL, 
\end{equation}
where if $\partial\Sigma$ is parameterized by the unit speed curve $\gamma(s)$, the function $k_g(s)$ is the inner product of $\nabla^{\mathbb{S}^3}_{\dot{\gamma}(s)}\dot{\gamma}(s)$ with the inner co-normal to $\Sigma$ at $\gamma(s)$.
\begin{conjecture}\label{WillmoreConj}
 If $\Sigma$ is a $C^2$ immersion of a M\"{o}bius band into $\mathbb{S}^3$ that spans $C$, then
 $$ \mathcal{W}(\Sigma)\geq \mathcal{W}(\bar{\tau}_{1,2}) = \mathrm{Area}(\bar{\tau}_{1,2}).$$
\end{conjecture}
This would imply a longstanding a conjecture of R. Kusner \cite{KusnerConj} for Klein bottles that contain and are symmetric about $C$:
\begin{conjecture}[Kusner \cite{KusnerConj} (1989)]
The Klein bottle in $\mathbb{S}^3$ of least possible Willmore energy is the Lawson Klein bottle $\tau_{1,2}$.
\end{conjecture}

As in Marques-Neves' proof of the Willmore Conjecture \cite{MN}, the proof of Theorem \ref{mainintro} implies sharp lower bounds of the Willmore energy of M\"obius bands. In contrast with Marques-Neves result we also obtain sharp lower bounds for conformal volume.\footnote{Their proof gives a lower bound on the Willmore energy for tori but not the conformal volume (which is smaller than $2\pi^2$ in most conformal classes \cite{MontielRos}, \cite{BryantTori}).} 
Let $\mathcal{M}$ denote the collection of minimal M\"obius bands in $\mathbb{S}^3$ with boundary $C$, and let $\mathcal{M}^*\subset\mathcal{M}$ denote those of infimal area $m_0$\footnote{We show in Section \ref{areaboundssection} that such a M\"obius band exists.}.  By Theorem \ref{mainintro}, any element in $\mathcal{M}^*$ has Morse index $2$, Euler number $\pm 2$ and is embedded.

For M\"obius bands we obtain:
\begin{thm}\label{willmoreband}
Let $\Sigma$ be a M\"obius band that is $C^2$-immersed in $\mathbb{S}^3$ and spanning $C$. Then
\begin{equation}\label{equality}
\mathcal{W}(\Sigma)\geq m_0, 
\end{equation}
with equality if and only if $\Sigma$ is a conformal dilate of an element in $\mathcal{M}^*$.
\end{thm}
In contrast, M. Pozzetta \cite{Pozzetta} showed that for orientable surfaces of fixed positive genus bounded by $C$, an immersion with infimal $L^2$ norm of its mean curvature (which is not a conformally invariant quantity in $\mathbb{R}^3$) is \emph{not} attained.
This reduces Kusner's conjecture for Klein bottles containing and symmetric with respect to $C$ to either Conjecture \ref{index2} or \ref{L2Conj}.
\begin{thm}\label{wd}
Let $\Sigma$ be a Klein bottle immersed in $\mathbb{S}^3$ containing the closed geodesic $C$ and invariant with respect to Schwarz reflection through $C$.  Then
\begin{equation}
\mathcal{W}(\Sigma)\geq 2m_0,
\end{equation}
with equality if and only if $\Sigma$ is a conformal dilate of the Schwarz reflection about $C$ of an element in $\mathcal{M}^*$.
\end{thm}

Recall, the notion of \emph{conformal volume} $\mathcal{V}_c(\Sigma)$ introduced by Li-Yau \cite{LY}:
\begin{equation}\label{liyau}
\mathcal{V}_c(\Sigma)=\sup_{v\in B^4}\mbox{Area}(F_v(\Sigma)), 
\end{equation}
where $\{F_v\}_{v\in B^4}$ is the $4$-ball of pure dilations in the conformal group of $\mathbb{S}^3$ -- Section \ref{conformalfamily} for the precise definition.  Li-Yau \cite{LY} proved for a closed immersion $\Sigma\subset\mathbb{S}^3$ there holds
\begin{equation}
\mathcal{W}(\Sigma)\geq \mathcal{V}_c(\Sigma).
\end{equation}
We obtain lower bounds on the  conformal volume:
\begin{thm} \label{ConfAreaLBThm}
If $\Sigma\subset\mathbb{S}^3$ is a $C^1$ immersion of a M\"obius band spanning $C$, then
\begin{equation}
\mathcal{V}_c(\Sigma)\geq m_0
\end{equation}
with equality if and only if $\Sigma$ is a conformal dilate of an element in $\mathcal{M}^*$.
\end{thm}


\subsection{Connections with closed surfaces in $\mathbb{S}^3$}
It is a consequence of the monotonicity formula that the closed minimal surface in $\mathbb{S}^3$ of lowest area is an equatorial two-sphere with area $4\pi$. In 2014, in their proof of the Willmore Conjecture, F.C. Marques and A. Neves \cite{MN} showed the Clifford torus was the minimal surface of second lowest area:

\begin{thm}[Marques-Neves \cite{MN} (2014)]\label{marquesneves}
The closed minimal surface of second lowest area in $\mathbb{S}^3$ is (up to ambient isometry) the Clifford torus, i.e., $\tau_{1,1}$
which has $\mbox{Area}(\tau_{1,1})=2\pi^2$.
\end{thm}

The proof by Marques-Neves of Theorem \ref{marquesneves} was variational and used Almgren-Pitts min-max theory (\cite{Almgren}, \cite{Pi}).  They considered a $5$-parameter ``canonical family" of surfaces associated to a minimal surface $\Sigma$ of genus at least $1$, with areas bounded above by that of $\Sigma$.  By optimizing over this family, they could obtain a minimal surface with area at most that of $\Sigma$ and Morse index at most $5$. By a result of F. Urbano \cite{Urbano}, the only minimal surfaces in $\mathbb{S}^3$ with index at most $5$ are great spheres and Clifford tori.  Using a topological degree argument, they showed they could not obtain a great sphere.  Our strategy in proving Theorem \ref{mainintro}, sketched in the next subsection, is inspired by their approach.

The infimal Willmore energy of a sphere immersed in $\mathbb{R}^3$ is $4\pi$. 
As the Willmore energy of a surface bounds from above the areas of elements in the $5$-parameter canonical family (without the minimality assumption), Marques-Neves were able to solve the Willmore Conjecture:
\begin{thm}[Marques-Neves 2014]
The Willmore energy of an immersed surface in $\mathbb{S}^3$ of positive genus  is bounded from below the Willmore energy of the Clifford torus.  
\end{thm}

Less is known about Willmore minimizers of fixed topological type. By a result of T. Banchoff \cite{Banchoff}, an immersed $\mathbb{RP}^2$ must contain a triple point, which, when combined with work of Li-Yau \cite{LY} forcing the infimal Willmore energy to be at least $12\pi$.  R. Bryant \cite{Bryantconf} and R. Kusner \cite{KusnerRp2} found examples realizing this number and studied their moduli space.  Among $\mathbb{RP}^2$ immersed in $\mathbb{R}^4$, Li-Yau \cite{LY} showed that the stereographic projection of the Veronese embedding with area $6\pi$ is optimal.  See also  \cite{BHM} for results for Klein bottles in $\mathbb{R}^4$.

\subsection{Sketch} Let us sketch some of the main ideas in this paper.  To any non-orientable surface $\Sigma$ with boundary the great circle $C$ and non-zero Euler number, we consider the $D^2$-family $\{\Sigma_{r,\theta}\}_{(r,\theta)\in D^2}$ of conformal dilates of the surface that preserve the boundary $C$ (as a set).  The areas in this ``canonical family" are all controlled from above by the Willmore energy or conformal volume of $\Sigma$.  We then apply a min-max argument to the canonical family to obtain an index (at most) $2$ minimal surface with less area than that of $\Sigma$.  Similar to the proof of the Willmore Conjecture, it is essential that the width not be equal to $2\pi$, i.e. that the min-max process does not merely produce a hemisphere of area $2\pi$ and thus only gives a trivial lower bound. The key point is that as $r\to 1^{-}$, the surfaces $\Sigma_{r,\theta}$ in the canonical family converge as varifolds to a hemisphere with boundary $C$ determined by the co-normal direction of $\Sigma$ at the point on $C$ at angle $\theta$ (see Figure \ref{picture}). Thus the boundary of the canonical family ``detects" the Euler number (in a similar way that the boundary of the Marques-Neves $5$-parameter family ``detects" the genus of the original surface). If it is non-zero, then if the width were $2\pi$ the $D^2$-canonical family could be homotoped to be near the space of hemsipheres bounded by $C$.  The space of such hemispheres is homeomorphic to $S^1$. This would allow us to build an extension to the disk $D^2$ of a map of non-zero degree from its boundary $\partial D^2$ to $S^1$, which is impossible.  We note that our canonical family differs from than that of Marques-Neves in the closed setting since there is not an extra parameter beyond the conformal ones. This completes the sketch in the case when the Euler number is non-zero.

  \begin{figure}\label{picture}
	\centering
	\resizebox{4in}{!}{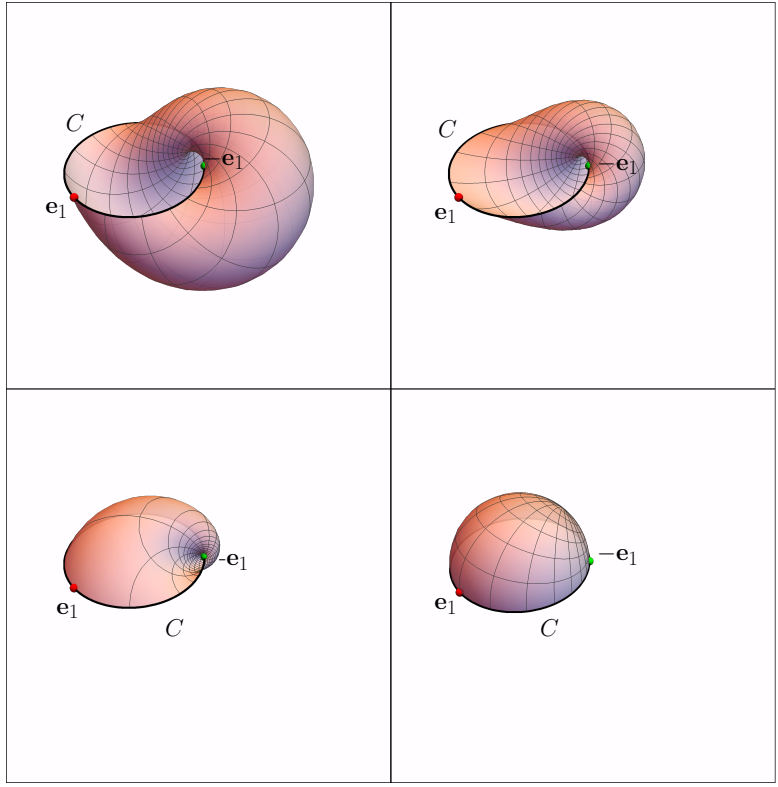}
	\caption{Stereographic projection of the conformal family of the Lawson band $\overline{\tau}_{1,2}$ pushing from $\mathbf{e}_1$ to $-\mathbf{e}_1$.  In particular, four stages of the surface $(\overline{\tau}_{1,2})_{r,0}$ as $r$ increases from $0$ (on upper LHS) to $1$ (on bottom RHS), where it coincides with a hemisphere bounded by $C$.}
	\label{ConfFamily}
\end{figure}

If instead the Euler number of the M\"obius band is zero (such as Boy's surface with a small disk removed) and the band has low conformal volume, we show that we can cap it off to obtain an immersed $\mathbb{RP}^2$ in $\mathbb{R}^3$ without adding additional triple points.  A result of T. Banchoff \cite{Banchoff} implies the band itself must then have had a triple point, implying the conformal volume of the band was in fact too large to begin with.

In a forthcoming artice \cite{BKMinmax}, we develop the min-max theory for M\"obius bands.  We use the min-max theory first developed by Colding-Minicozzi \cite{CM1}, who obtained minimal $2$-spheres from non-trivial sweepouts of homotopy three-spheres. In particular, we use its  extension to a min-max process for tori by X. Zhou \cite{Zhoutorus} (where one allows the conformal structure to vary) and to the context of maps with boundaries constrained to a submanifold by L. Lin, A. Sun and X. Zhou \cite{LSZ}.   In adapting their work, we observe some novel phenomena related to the non-orientability, including an analysis of $\mathbb{RP}^2$ bubbles and a degree formula mod 2 for boundary parameterizations of limits. 

The organization of this paper is as follows. In Section \ref{general} we give basic observations on low-area minimal surfaces with boundary $C$.  In Section \ref{conformalfamily} we introduce the conformal family and deduce its non-triviality when the Euler number is non-zero.  In Section \ref{sectionminmax} we discuss the min-max theorem we will need and its consequences.  In Section \ref{areaboundssection} we prove Theorem \ref{mainintro}.  In Section \ref{applicationsection} we prove various results about conformal volume and Willmore energy. Section \ref{appendix} introduces a mollification procedure.  In Section \ref{problems} we collect open problems and directions for research related to constructions in this paper. 
\\
\newline
\emph{Acknowledgements}:  J.B was partially supported by the NSF grant DMS-2203132.  D.K. was partially supported by NSF grant DMS-2405114.

\subsection{Notation}
Throughout this paper, $\mathbb{S}^3$ is endowed with its standard round metric induced from $\mathbb{R}^4$.  Denote by $G_2(\mathbb{S}^3)$ the Grassmannian bundle of unoriented $2$-planes in $\mathbb{S}^3$.  A $2$-varifold is a Radon measure on $G_2(\mathbb{S}^3)$.  We denote by $\mathcal{V}_2(\mathbb{S}^3)$ the space of $2$-varifolds and for $V\in\mathcal{V}_2(\mathbb{S}^3)$ we denote by $||V||(\mathbb{S}^3)$ its total mass.  
If $V,W\in\mathcal{V}_2(\mathbb{S}^3)$, then following Pitts (\cite{Pi}, pg. 66) define the $\mathbf{F}$-metric by:
\begin{equation}\label{defoff}
\mathbf F(V,W)
=
\sup
\left\{
|V(f)-W(f)|:
f\in C^1(G_n(\mathbb{S}^3)),
\ |f|\le 1,\ 
\operatorname{Lip}(f)\le 1
\right\}.
\end{equation}
The $\mathbf{F}$-metric induces the weak topology on $\mathcal{V}_2(\mathbb{S}^3)\cap\{ V\;|\; ||V||(\mathbb{S}^3)\leq a\} $ for each $a>0$.
For a surface $\Sigma$ immersed in $\mathbb{S}^3$, we denote by $|\Sigma|\in\mathcal{V}_2(\mathbb{S}^4)$ the corresponding $2$-varifold.

\section{Surfaces bounded by a great circle}\label{general}
We may describe the unit sphere $\mathbb{S}^3\subset\mathbb{R}^4=\mathbb{C}^2$ by writing
\begin{equation}
\mathbb{S}^3=\{(z,w)\in\mathbb{C}^2\;|\; |z|^2+|w|^2=1\}.
\end{equation}
Let us denote the great circle
\begin{equation}
C=\{(z,w)\in\mathbb{S}^3\;|\; w=0\}.
\end{equation}
In this section, we record some properties of surfaces spanning $C$.  

There is an $\mathbb{S}^1$-family of hemispheres $\{H_\phi\}_{\phi\in [0,2\pi]}$ in $\mathbb{S}^3$ with boundary $C$ given by
\begin{equation}
H_\phi =
\{ (z,w) \in S^3 :
\operatorname{Im}(e^{-i\phi} w)=0,\;
\operatorname{Re}(e^{-i\phi} w)\ge 0 \},
\qquad
0 \le \phi \le 2\pi .
\end{equation}
Let us denote the collection of such hemispheres:
\begin{equation}\label{identify}
\mathcal{H} =\{H_\theta\;|\; \theta\in [0,2\pi]\}.
\end{equation}
There holds
\begin{equation}
\mbox{Area}(|H_\phi|)=2\pi\mbox{ for each } \phi\in\mathbb{S}^1.
\end{equation}

\subsection{Spanning surfaces and area bounds}
Since we will be working with objects less regular than smoothly embedded submanifolds, the following definition is useful:
\begin{defn}
Let $\Sigma$ be a compact surface with one boundary component that is non-empty and oriented.  If $u:\Sigma\to\mathbb{S}^3$ is a $W^{1,2}\cap C^0$ map where $u(\partial\Sigma)\subset C$ and $u|_{\partial\Sigma}:\partial\Sigma\to C$ has degree $1$, then $u(\Sigma)$ \emph{spans} $C$ (or for ease of notation, $u$ spans $C$).
\end{defn}
Here the degree is the topological degree of the continuous map $u|_{\partial \Sigma}$ and we have fixed an orientation on $C$.  We note that if $u|_{\partial \Sigma}$ is injective, then, up to reversing orientation, $u$ spans.

We first observe the following area bound for spanning surfaces:

\begin{lemma}\label{AreaLowBnd}
Let $\Sigma$ be a compact surface with one boundary component.  If $u:\Sigma\to \mathbb{S}^3$ is a $W^{1,2}\cap C^0$ map so that $u$ spans $C$, then, $\mathrm{Area}(u)$, the (parameterized) area of $u$, i.e., $\mathrm{Area}(|u(\Sigma)|)$, the mass of the asociated varifold as in \eqref{VarifoldToMap}, is at least $2\pi$. 
\end{lemma}
\begin{proof}
By Lemma \ref{varcont} and Proposition \ref{mollprop}, for any $\epsilon>0$, one can find a $u_\epsilon\in C^\infty(\Sigma)$ so that $u_\epsilon(\Sigma)$ spans $C$ and the parameterized area of $u_\epsilon$ and $u$ differ by less than $\epsilon$.  Note that $u_\epsilon(\Sigma)$ is a rectifiable set.

Let $\mathcal{C}$ denote the space of (oriented) great circles in $\mathbb{S}^3$ -- i.e., the space of oriented closed geodesics.  If $H$ is a hemisphere bounded by $C$, and $c\in \mathcal{C}$ is disjoint from $C$, then $c$ either $c$ is linked with $C$ and $c$ meets $H$ once or $c$ is not linked with $C$ and is disjoint from $H$.  That is for $c$ disjoint from $C$ the intersection number between $c$ and $H$ satisfies $n(c,H)=0$ when $c$ is unlinked with $C$ and $1$ when $c$ is linked. It follows from the Cauchy-Crofton formula in the sphere \cite{santaloIntegralGeometryGeometric2004} that
$$
2\pi=\mathcal{H}^2(H)=\int_{\mathcal{C}} n(c,H) d\mu_{\mathcal{C}}
$$
where here $\mu_{\mathcal{C}}$ is the appropriate measure on $\mathcal{C}$ and we use the set of elements of $\mathcal{C}$ intersecting $C$ has $\mu_C$ measure zero.
By the spanning condition,  $c\in \mathcal{C}$ is disjoint from $u(\partial \Sigma)$ if and only if it is disjoint from $C$ and is linked with $u|_{\partial \Sigma}$ if and only if $c$ is linked with $C$.  It follows that for $c$ disjoint from $C$,  $n(c,u_\epsilon(\Sigma))\geq n(c,H)$ and hence, as $u_\epsilon(\Sigma)$ is rectifiable the area of this set satisfies
$$
\mathcal{H}^2(u_\epsilon(\Sigma))=\int_{\mathcal{C}} n(c,u_\epsilon) d\mu_{\mathcal{C}} \geq  2\pi.
$$
By the area formula \cite[pg. 46]{SimomBook}, it follows that the parameterized area of $u_\epsilon$, , i.e., $\mathrm{Area}(|u_\epsilon(\Sigma)|)$ is at least $2\pi$ and so $\mathrm{Area}(|u(\Sigma)|)\geq 2\pi-\epsilon$.  Since $\epsilon>0$ is arbitrary the claim follows.
\end{proof}

We have the following:
\begin{lemma}\label{orientable}
Let $\Sigma$ be an orientable surface with one boundary component.  We have the following:
\begin{enumerate}
\item A branched \emph{minimal} immersion $u:\Sigma\to \mathbb{S}^3$ with $u({\partial \Sigma})\subset C$ has area at least $2\pi$ with equality if and only if $\Sigma$ is a disk and $u$ is the embedding of a hemisphere with boundary $C$.
\item  A branched \emph{minimal} immersion $u:\Sigma\to \mathbb{S}^3$ with $u(\partial \Sigma)\subset C$ that is not a hemisphere has area at least $4\pi$.
\item A branched minimal immersion of a disk spanning $C$ is a hemisphere with area $2\pi$.
\item If $u:\Sigma\to \mathbb{S}^3$ is a $W^{1,2}\cap C^0$ map so that $u$ spans $C$, then $\mathrm{Area}(u)$, the (parameterized) area of $u$ is at least $2\pi$. 
\end{enumerate}
\end{lemma}

\begin{proof}
For (1) and (2) observe by Proposition 3.1 in \cite{Lawson}, we can Schwarz reflect $\Sigma$ about $C$ to obtain a closed minimal surface $\Sigma'$ with $\mbox{Area}(\Sigma')=2\mbox{Area}(\Sigma)$.   By a result of Hardt-Simon \cite{HS}, if $\Sigma$ is embedded it is a hemisphere.  Thus we can assume without loss of generality that $\Sigma$ is immersed as is $\Sigma'$.  It follows that $\Sigma'$ has a point of density at least two, implying its area is at least $8\pi$.  To see this,  let $C(\Sigma')$ denote the cone over $\Sigma'$ in $\mathbb{R}^4$.  Then there exists a sequence of points $p_i\in\mathbb{R}^4$ with $p_i\to 0$ so that 
\begin{equation}
\Theta(C(\Sigma'), p_i)\geq 2.
\end{equation}
By the upper semicontinuity of density, this implies that
\begin{equation}
\Theta(C(\Sigma'), 0)\geq \Theta(C(\Sigma'), p_i)\geq 2. 
\end{equation}
Since 
\begin{equation}
\frac{\mathrm{Area}(\Sigma')}{4\pi}=\frac{\mathrm{Area}(\Sigma)}{2\pi}= \Theta(C(\Sigma'), 0)\geq 2, 
\end{equation}
we conclude
\begin{equation}
\mbox{Area}(\Sigma)\geq 4\pi.
\end{equation}  This completes the proof of (1) and (2). Item (3) follows from the classification of minimal two-spheres in $\mathbb{S}^3$. 
 Item (4) is an immediate consequence of \ref{AreaLowBnd}.  
\end{proof}
We also have for non-orientable surfaces: 

\begin{lemma} \label{nonorientable} Let $\Sigma$ be a non-orientable surface with one boundary component. We have the following:
\begin{enumerate}
\item A branched \emph{minimal} immersion $u:\Sigma\to \mathbb{S}^3$ with $u({\partial \Sigma})\subset C$ has area at least $4\pi$.
\item If $u:\Sigma\to\mathbb{S}^3$ is a  $W^{1,2}\cap C^0$ map  spanning $C$ then 
the parameterized area of $u$, is at least $2\pi$.
\item For all $\varepsilon>0$ there exists $\delta>0$ so that if $u:\Sigma\to\mathbb{S}^3$ 
is a $W^{1,2}\cap C^0$ map where $u$ spans $C$ and $\mbox{Area}(u)\leq 2\pi+\delta$, then 
\begin{equation}
\mathbf{F}(|u(M)|, \mathcal{H})<\varepsilon.
\end{equation}
\end{enumerate}
\end{lemma}
\begin{proof}
To see (1), observe that we can Schwarz reflect $\Sigma=u(\Sigma)$ about the geodesic $C$ to obtain a closed minimal non-orientable surface $\tilde{\Sigma}$.  Since closed non-orientable surfaces do not embed in $\mathbb{S}^3$, it follows that $\tilde{\Sigma}$ has a point of density at least $2$ from which the conclusion follows as in the previous lemma.  
Item (2) immediately follows from Lemma \ref{AreaLowBnd}.
To see (3), we first note that by Lemma \ref{varcont} and Proposition \ref{mollprop} it suffices to prove the result assuming $u\in C^\infty(\Sigma)$.  Now, suppose, for some $\varepsilon>0$, there exists no such $\delta$.  Then we obtain a sequence $u_i\in C^\infty(\Sigma)$ with $u_i$ spanning $C$ and so $\mbox{Area}(u_i)\to 2\pi$ but 
\begin{equation}\label{contra}
\mathbf{F}(|u_i(\Sigma)|,\mathcal{H})\geq \varepsilon.
\end{equation}
Let $\Sigma_i=(u_i)_{\#} \Sigma$ be the (rectifiable) mod 2 flat chain obtained by pushing forward $\Sigma$ by $u_i$.  The spanning hypothesis implies $(u_i)_\# \partial \Sigma=C$ and so $\partial \Sigma_i= C$.  
Moreover, the mass,  $\mathbb{M}(\Sigma_i)$, of $\Sigma_i$ is bounded by $\mbox{Area}(u_i)$ and so \begin{equation}\limsup_{i\to \infty } \mathbb{M}(\Sigma_i)\leq 2\pi.\end{equation}

By \cite{Hardt}, the area-minimizing mod 2 flat chain spanning $C$ is smooth and so invoking either Lemma \ref{orientable} item (1) or item (1) above, this minimizer is a hemisphere.  It follows that, up to passing to a subsequence, $\Sigma_i$ converge in the flat norm to $H$ for some hemisphere $H$ and their masses converge to the mass of $H$.  Up to passing to a further supsequence, we may assume the varifolds $|u_i(\Sigma)|$ converge to some varifold $U_\infty$. 

It follows from \cite[Corollary 1.3]{liuWeakConvergenceVarifolds2022} that the associated varifolds, $V_i$, to the $\Sigma_i$ converge to the varifold, $|H|$ associated to $H$.  It is clear that $|u_i(\Sigma)|\geq V_i$ and so $U_\infty\geq |H|$.  However, since the areas of both varifolds are equal, equality holds which contradicts \eqref{contra}.
\end{proof}

We show that a branched minimal immersion $u: \Sigma \to \mathbb{S}^3$ that spans $C$ which has parameterized area below $6\pi$ is a smooth embedding.
\begin{prop}\label{NoBPProp}
    If $u: \Sigma \to \mathbb{S}^3$ is a branched minimal immersion that spans $C$ and $\mathrm{Area}(|u(\Sigma)|)<6\pi$, then $u$ is a smooth embedding.
\end{prop}
\begin{proof}
Let $u':\Sigma'\to \mathbb{S}^3$ be the branched immersion obtained by Schwarz reflecting across $C$.  It follows that $|u'(\Sigma')|$ is a stationary integral varifold.  By the monotonicity formula applied to $C(u'(\Sigma'))$, for any $p\in \mathbb{R}^4$
$$
\frac{1}{4\pi} \mathrm{Area} (|u'(\Sigma')|)=\frac{1}{2\pi} \mathrm{Area}(|u(\Sigma)|)\geq \Theta (C(|u'(\Sigma')|),p).
$$

As $u$ has degree one, any small curve linked once with $C$ is also linked once with $u|_{\partial \Sigma}$ hence, by $\mathbb{Z}_2$ intersection theory, a generic such curve meets the image of $u$ an odd number of times.  It follows that any boundary branch points have odd order and hence if one exists there is a point of density at least three on $\Sigma'$.  In particular, if there is a boundary branch point, the area of $|u'(\Sigma')|$ is at least $6\pi$ which is precluded by the area assumption.
For the same reason, $u'$ can't have a self-intersection point on $C$.  That is the image of $u'$ and hence of $u$ is embedded and smooth near $C$

To deal with interior branch points, we think of $|u(\Sigma)|$ as a stationary integral varifold in $\mathbb{S}^3 \setminus C$.   We consider the cone $C(|u(\Sigma)|)$ which is then a stationary varifold in $\mathbb{R}^4$.  In fact, since $\Sigma$ is embedded and smooth near $C$ the boundary, $C(|u(\Sigma)|)$ is stationary away from a two plane, $P$, which is the cone over $C$. Moreover,  $C(|u(\Sigma)|)$ is regular on $P$ away from the origin. In particular, $C(|u(\Sigma)|)$ is regular enough to apply the extend monotonicity formula of \cite{EWW, whiteMeanCurvatureFlow2021}.

It follows that for any $p\in \mathbb{R}^4\setminus P$
$$
\frac{1}{4\pi} (\mathrm{Area} (|u(\Sigma)|)+\frac{1}{2} \geq \Theta(C(|u(\Sigma)|, p).
$$
This means that if there is an interior branch point or point of self-intersection, 
$$
\mathrm{Area}(|u(\Sigma)| \geq 6\pi.
$$
Hence, any interior singularity is precluded by the area bound and so we conclude that $u$ is a smooth minimal embedding.
\end{proof}

When $u$ is an immersion without branch points one can directly adapt the argument of \cite{LY} (Fact (iii) pg. 271) to setting of surfaces spanning $C$ and avoid any use of geometric measure theory to prove area embeddedness when the area is small.
\begin{lemma}
 \label{EmbNearCLem}
If $u:\Sigma\to \mathbb{S}^3$ is an immersed minimal surface spanning $C$ and $
\mathrm{Area}(|u(\Sigma)|)<6\pi,
$
then $u$ is an embedding.
\end{lemma}
\begin{proof}
  First note that if $u$ is not an embedding near $C$, then the spanning condition ensures there is a point $p\in C$ so that either $u^{-1}(p)$ contains at least three points  or $u^{-1}(p)$ contains at least two points with only one on $\partial N$.  In either case, it follows that if $\Sigma'$ is the surface obtained by doubling $\Sigma$ using Schwarz reflection, then $p$ is at least a triple point of $\hat{\Sigma}$. Hence, by \cite{LY}, $\mathrm{Area}(\hat{\Sigma})\geq 12 \pi$.  As such, when $u$ is not an embedding, any point of self-intersection, $p$, satisfies $p \in \Sigma \setminus C$.

  Consider the family of conformal automorphisms, $\phi_t:\mathbb{S}^3\to \mathbb{S}^3$ so that $\phi_t(\pm p)=\pm p$, $\phi_0$ is the identity map and $\lim_{t\to 1^-} \phi_t(q)=-p$ for $q\neq p$.  Following \cite{LY}, we may use the Gauss equations and Gauss-Bonnet theorem to obtain, for $\Sigma_t= \phi_t(\Sigma)$,
$$
\int_{\Sigma_t}  1+\frac{1}{4}|\mathbf{H}_{\Sigma_t}|^2 dA+\int_{\partial \Sigma_t} k_{\partial  \Sigma_t} dL = \frac{1}{2}\int_{\Sigma_t} |\mathring{\mathbf{A}}_{\Sigma_t}|^2 dA+2\pi(g(\Sigma_t)-1) 
$$
where here $\mathring{\mathbf{A}}_{\Sigma_t}$ and  $k_{\partial  \Sigma_t}$ are, respectively, the trace free second fundamental form of $\Sigma_t$ and geodesic curvature of $\partial \Sigma_t$.

As $\partial  \Sigma_t= \phi_t(\partial \Sigma)$ is a round circle and $u$ is an embedding on $\partial \Sigma$,  the following estimate holds:
$$
\left| \int_{\partial \Sigma_t } k_{\partial  \Sigma_t} dL\right| \leq 2\pi.
$$
The conformal invariance of the right hand side implies that
$$
-2\pi+\int_{\Sigma_t}  1+\frac{1}{4}|\mathbf{H}_{\Sigma_t}|^2 dA\leq \mathrm{Area}(\Sigma).
$$
As $p$ is a point of self-intersection of $\Sigma$ we have
$$
8\pi \leq \liminf_{t\to \infty} \int_{\Sigma_t}  1+\frac{1}{4}|\mathbf{H}_{\Sigma_t}|^2 dA
$$
from which the claim follows.
\end{proof}
\subsection{Index bounds}\label{indexsection}
In this section we discuss basic facts about the Morse index of non-orientable minimal surfaces with boundary $C$.

We suppose $\Sigma$ is a non-orientable minimal surface bounded by $C$.  We allow $\Sigma$ to be immersed which we take to mean that there is a minimal immersion of a non-orientable surface with boundary $\phi:N\to \Sigma\subset \mathbb{S}^3$ with $\phi: \partial N \to C$ an embedding. Let $\pi: \Sigma'\to \Sigma$ by the oriented double cover of $\Sigma$ and let $\zeta: {\Sigma'}\to {\Sigma'}$ be the deck transform.  If $\Sigma$ is immersed this means there is an immersion $\phi':N'\to \Sigma'$ with $\pi:N'\to N$ the orientation double cover and $\phi\circ \pi=\phi'$ and $\zeta:N'\to N'$.  We note that ${\Sigma'}$ has two boundary components, $\partial {\Sigma'}={C'}_+\cup {C'}_-$ and $\zeta(C_\pm')=C_\mp'$. 
 
 We say a function $u$ on $\Sigma'$ is \emph{even} if $u\circ \zeta=u$ and is \emph{odd} if $u\circ \zeta =-u$.   On $\Sigma'$ there is a well defined choice of unit normal $\mathbf{n}$ that satisfies $\mathbf{n}\circ \zeta=-\mathbf{n}$.    Clearly, for any odd function $u$, there is a unique section $\mathbf{s}$ of the normal bundle of $\Sigma$ so that $u \mathbf{n}=\mathbf{s}\circ \pi$. The converse is also true, for any such section, there is a unique choice of odd function on $\Sigma'$.  These concepts are adapted to the immersed setting in the obvious way.

Let us denote by $L_{\Sigma}$ the stability operator on on the normal bundle of $\Sigma$ which we express as
$$
L_{\Sigma}=\Delta^\perp_{\Sigma} +|A_{\Sigma}|^2 +2
$$ 
and we note that while the second fundamental form is not globally well defined on $\Sigma$, its squared norm $|A_{\Sigma}|^2$ is.
One readily verifies that if $\mathbf{e}_i^\perp$ is the orthogonal projection of the coordinate vector fields $\mathbf{e}_i$ to the normal bundle, then
$$
-L_{\Sigma} \mathbf{e}_i^\perp=-2 \mathbf{e}_i^\perp.
$$
In order to study the spectrum of the stability operator on $\Sigma$, it is convenient to consider the spectrum of the operator
$$
L_{\Sigma'}=\Delta_{\Sigma'} +|A_{\Sigma'}|^2 +2
$$
restricted to odd functions on $\Sigma'$.  The (Dirichlet) Morse index of $\Sigma$ is the dimension of the space of odd Dirichlet eigenfunctions of $-L_{\Sigma'}$ which have negative eigenvalue.  In this setting one verifies that, after choice of unit normal, $\mathbf{n}$, $n_i=\mathbf{n}\cdot \mathbf{e}_i$ are odd, satisfy 
\begin{equation}\label{eigenvalue}
-L_{\Sigma'} n_i=-2 n_i\end{equation}
and, on $C'=\partial \Sigma'$,
\begin{equation}\label{dirichlet}
n_1=n_2=0 \mbox{ and } \partial_{\nu'} n_3=\partial_{\nu'} n_4=0.
\end{equation}
That is, $n_1$ and $n_2$ are odd Dirichlet eigenfunctions and $n_3$ and $n_4$ are odd Neumann eigenfunctions.  In summary, if
$$
\mathcal{U}'_-=\set{u\in C^\infty(\Sigma') : - L_{\Sigma'}u =\lambda u, u|_{\partial \Sigma'}=0, u\circ \zeta =-u, \lambda<0}
$$
is the set of odd Dirichlet eigenvalues of $-L_{\Sigma'}$ with negative eigenvalue, then $n_1,n_2\in \mathcal{U}'$ and the Dirichlet Morse index of $\Sigma$, can be computed as
$$
\mathrm{ind}_D(\Sigma)=\dim \mathcal{U}'_-.
$$

We observe the following lower bound on index for non-flat $\Sigma$ (see J. Simons' Theorem 5.1.1 \cite{Simons}):
\begin{prop}\label{IndLowBndProp} If $\Sigma$ is not a hemisphere, then
	$$\mathrm{ind}_D(\Sigma)\geq 2.$$
\end{prop}
\begin{proof}
	If $\mathrm{ind}_D(\Sigma)\leq 1$, then there must be a non-trivial linear relation
	$$
	a \mathbf{e}_1^\perp+b\mathbf{e}_2^\perp=0, a^2+b^2 =1.
	$$
	Taking the cone over $\Sigma$ gives a minimal cone $\mathcal{C}$ in $\mathbb{R}^4$ bounded by the plane $x_3=x_4=0$. This cone, $\mathcal{C}$ contains the line spanned by $a \mathbf{e}_1+b\mathbf{e}_2$, i.e., the cone splits off a line.  One has $\Sigma=\mathcal{C}\cap \mathbb{S}^3$ is smooth only when $\mathcal{C}$ is a half-plane in which case $\Sigma$ would be an hemisphere.
\end{proof}

\section{Conformal family}\label{conformalfamily}
Throughout this section, let $u:M\to\mathbb{S}^3$ be a $C^1$ immersion of a M\"obius band with \begin{equation}\label{good}u|_{\partial M}:\partial M\to C\end{equation} an embedding and, even more, the integral $2$-varifold $|u(M)|$ has density $\frac{1}{2}$ on $C$.

\begin{defn}\label{eulerdef} If $\Sigma=u(M)$ the curve $\partial T_\epsilon(C)\cap \Sigma$ is a $(1,k)$ torus knot for some integer $k\in\mathbb{N}$.  The winding number, or \emph{Euler number} $e(\Sigma)$ is defined to be $k$. 
\end{defn}
If $\Sigma$ is an an embedding of a M\"obius band, then by a result of Whitney \cite{Whitney} and Massey \cite{Ma}, $e(\Sigma)=\pm 2$ -- see also \cite{Yasuhara}. The immersed Lawson band $\overline{\tau}_{1,2k}$ has $e(\overline{\tau}_{1,2k})=2k$ for each $k\geq 1$. To obtain an example with Euler number zero, remove a small disk from an immersion of $\mathbb{RP}^2$ in $\mathbb{S}^3$ (such as Boy's surface).  On the other hand it is easy to see:

\begin{lemma}\label{iseven}
If $|\Sigma|$ has density $\frac{1}{2}$ on $C$, then $e(\Sigma)$ is even. 
\end{lemma}
\begin{proof}
Let $N$ denote a small enough tubular neighborhood of $C$, and let $T=\overline{\mathbb{S}^3\setminus N}$ denote the complementary Heegaard torus.  Then by the density one hypothesis, $ \Sigma\cap \partial T$ consists of a simple closed curve winding once around the meridian curve in $\partial T$ and $k$ times around the longitudinal curve.  By the density $\frac{1}{2}$ hypothesis, $\Sigma\cap \partial N$ bounds an immersed surface $\Gamma=\Sigma\cap T$ in $T$.  Thus, $\partial\Gamma$ as a mod $2$ cycle is trivial in $H_1(T;\mathbb{Z}_2)$.  But $H_1(T;\mathbb{Z}_2)$ is generated by the longitidunal class of $\partial T$, and thus $k$ must be even.
\end{proof}

\subsection{Construction of conformal family}
In this section we associate to a map $u$ as above a $2$-parameter ``canonical family" of immersions of $M$ into $\mathbb{S}^3$ parameterized by the interior of the unit disk $D^2$.

For $v\in B^4$, consider the conformal diffeomorphism \begin{equation}F_v:\mathbb{S}^3\to\mathbb{S}^3\end{equation} given by 
\begin{equation}\label{conf1}
F_v(x) := \frac{1-|v|^2}{|x-v|^2}(x-v) - v.
\end{equation}
For $v\neq 0$, the map $F_v$ fixes $\pm\frac{v}{|v|}\in\mathbb{S}^3$ and pushes everything else toward $-\frac{v}{|v|}\in\mathbb{S}^3$.  For $(r,\theta)\in D^2$, we define 
\begin{equation}
v:D^2\to B^4
\end{equation}
by 
\begin{equation}
v(r,\theta) = r(\cos(\theta),\sin(\theta),0,0).
\end{equation}
It is easy to see from \eqref{conf1} that if $x=(x_1,x_2,0,0)$, then $F_{v(r,\theta)}(x_1,x_2,0,0)$ has zero for its third and fourth components.  Since $F_v^{-1}=F_{-v}$, we thus get 
\begin{equation}\label{boundarystll}
F_{v(r,\theta)}(C)=C\mbox{ for all } (r,\theta)\in\mbox{int}(D^2).
\end{equation}

For $(r,\theta)\in\mbox{int}(D^2)$ we define the mappings
\begin{equation}
u_{r,\theta}=F_{v(r,\theta)}\circ u :M\to \mathbb{S}^3.
\end{equation}
Let us write
\begin{equation}
\Sigma_{r,\theta}=u_{r,\theta}(M).
\end{equation}
Note that
\begin{equation}
u_{0,\theta}=u,
\end{equation}
for all $\theta\in S^1$ and we obtain a family $\{u_{r,\theta}\}_{(r,\theta)\in\mbox{int}(D^2)}$ varying continuously in the smooth topology with the property (by \eqref{boundarystll}) that $\partial\Sigma_{r,\theta}=C$ for all $(r,\theta)\in\mbox{int}(D^2)$.  

The following definitions will be useful in describing the behavior of the family $\{u_{r,\theta}\}_{(r,\theta)\in\mbox{int}(D^2)}$ as $r\to 1^-$.
\begin{defn}
For $p\in C$  and $v\in T_p \mathbb{S}^3\setminus T_p C$, let $H(v)\in\mathcal{H}$ be the unique element with \begin{equation}v\in T^{+}_p H(p,v).  \end{equation}  For $p\in C$ and $\Sigma$ a surface with boundary $C$, let us denote by $\nu_p(\Sigma)$ the inner unit co-normal to $\Sigma$ at $p$. 
\end{defn}

For each $\theta\in [0,2\pi]$ define the constant map
\begin{equation}
L_\theta:M\to\mathbb{S}^3
\end{equation}
given by 
\begin{equation}
L_\theta(x) = (\cos(\theta),\sin(\theta), 0,0)\in\mathbb{S}^3.
\end{equation}

Recall we can identify $\mathcal{H}$, the space of hemispheres with boundary $C$ with $\mathbb{S}^1$ by \eqref{identify}.  In particular, for any $v\in\mathbb{S}^1$, denote by $H_v$ the corresponding element in $\mathcal{H}$ and let $\phi:\mathcal{H}\to\mathbb{S}^1$ be the inverse $\phi(H_v)=v$.  In the following, note that, identifying $\mathbb{S}^3$ with $\partial B^4$, $\{v(1,\theta)\}_{\theta\in [0,2\pi]}$ parameterizes the geodesic $C\subset\mathbb{S}^3$.   

The family $\{u_{r,\theta}\}_{(r,\theta)\in\mbox{int}(D^2)}$ extends continuously in the sense of varifolds to $\partial D^2$ and its boundary ``detects" the Euler number:

\begin{prop}[Canonical family] \label{limits}
Suppose $u:M\to\mathbb{S}^3$ is an immersion with boundary $C$ so that $|u(M)|$ has density $\frac{1}{2}$ on $C$. For any $\theta_0\in [0,2\pi]$, the following hold:  
\begin{enumerate}
    \item In the $C^\infty_{loc}(M\setminus u^{-1}(v(1,-\theta_0)),\mathbb{S}^3)$ topology there holds \begin{equation}\lim_{(r,\theta)\to (1^-,\theta_0)}u_{r,\theta}= L_{-\theta_0}.\end{equation}
\item In the sense of varifolds, there holds
\begin{equation}\label{limit}
\lim_{(r,\theta)\to (1^-,\theta_0)} |\Sigma_{r,\theta}|= |H(\nu_{v(1,\theta_0)}(\Sigma))|=:|E_\Sigma(\theta_0)|,
\end{equation}
where $E_\Sigma(\theta_0)\in\mathcal{H}$.
\item The degree of the map $F_\Sigma:\partial D^2\to \mathbb{S}^1$ given by 
\begin{equation}
F_\Sigma(\theta):=\phi\circ E_\Sigma(\theta)=\phi\circ H( \nu_{v(1,\theta)}(\Sigma)),
\end{equation}
is equal to the Euler number $e(\Sigma)$.
\item The family $\{|\Sigma_{r,\theta}|\}_{(r,\theta)\in\mbox{int}(D^2)}$ extends continuously in the sense of varifolds to $\partial D^2$.
\end{enumerate}

\end{prop}
\begin{rmk}
Note in (1) that $u^{-1}$ is well-defined on $C$ by our assumptions \eqref{good} on $u$.
\end{rmk}
\begin{rmk}
Note that the canonical family can also be defined associated to any non-orientable surface $\Sigma$ with boundary $C$ and density $\frac{1}{2}$ on $C$, not necessarily a M\"obius band.
\end{rmk}
\begin{proof}
Since $u$ is an embedding near $\partial M$ and the density of $|u(M)|$ is equal to $\frac{1}{2}$ on $C$, we obtain that for small enough $\varepsilon>0$ $u(M)\cap T_\varepsilon(C)$ is an embedded submanifold, where $T_\varepsilon(C)$ is a tubular neighborhood of $C$.  The limit in item (2) follows easily.  As $u$ is a $C^1$ embedding near $\partial M$, the tangent plane to it at each point $p\in C$ varies continuously and is independent of the sequence of dilations.  This implies (2) and (4). Item (3) is clear from Definition \ref{eulerdef}.  Item (1) is also immediate.
\end{proof}
\subsection{Homotopy class of the conformal family}\label{homotopyclass}
We define the variational class as follows.   Let us first set 
\begin{equation}\Gamma=C^0\bigl(\mbox{int}(D^2), W^{1,2}\cap C^0(M,\mathbb{S}^3)\bigr).\end{equation}


Let us define the relative variational class with respect to $\Sigma$ which enforces a Dirichlet boundary condition at $\partial D^2$ given by the family of disks $\{E_\Sigma\}_{\theta\in [0,2\pi]}$:

\begin{equation}\label{homotopydef}
\Omega_\Sigma =
\left\{
\gamma \in \Gamma
\;\middle|\;
\begin{aligned}
&\gamma(p,*)|_{\partial M} \subset C
    \text{ for all } p\in\operatorname{int}(D^2),\\
&\gamma(p,*)|_{\partial M}:\partial M\to C
    \text{ has degree } 1
    \text{ for all } x\in\operatorname{int}(D^2),\\
&\gamma(re^{i\theta},*)
    \text{ varifold converges to the disk } |E_\Sigma(\theta_0)|
    \text{ as } r\to 1^-\\
&\mbox{ and } \theta\to \theta_0\mbox{ for each } \theta_0\in[0,2\pi].
\end{aligned}
\right\}.
\end{equation}

By the assumptions on $u$ and Proposition \ref{limits}, $\gamma_0:=\{u_{r,\theta}\}_{(r,\theta)\in D^2}\in\Omega_\Sigma$.
Let $\Pi_\Sigma\subset \Omega_\Sigma$ denote the homotopy class of the canonical family $\gamma_0$ defined as follows.   The sweepout $\gamma\in \Omega_\Sigma$ is contained in $\Pi_\Sigma$ if and only if there exists $\tilde{\gamma}_t(p,*)\in C^0(\mbox{int}(D^2)\times [0,1],C^0\cap W^{1,2}(M,\mathbb{S}^3))$ so that
\begin{enumerate}
\item $\tilde{\gamma}_0(p,*)=\gamma_0(p,*)\mbox{ for } p\in\mbox{int}(D^2)$
\item  $\tilde{\gamma}_1(p,1)=\gamma(p,*)\mbox{ for } p\in\mbox{int}(D^2)$,
\item $\tilde{\gamma}_t\in \Omega_\Sigma$ for all $t\in [0,1]$.  
\item If $t\to t_0$ and $re^{i\theta}\to e^{i\theta_0}$ then $|\tilde{\gamma}_t(re^{i\theta},*)|$ converges as varifolds to the hemisphere $|E_\Sigma(\theta_0)|$.\label{condition}
\end{enumerate}

Let us define the ``width" $w(\Pi_\Sigma)$ to be
\begin{equation}
w(\Pi_\Sigma)=\inf_{\Phi\in \Pi_\Sigma}\sup_{p\in D^2}\mbox{Area}(\Phi(p,*)).
\end{equation}
A sequence of sweepouts $\{\gamma_n\}_{n=1}^\infty\in\Pi_\Sigma$ is a \emph{minimizing sequence} if \begin{equation}\sup_{p\in D^2}\mbox{Area}(\gamma_n(p,*))\to w(\Pi_\Sigma).\end{equation}  
A sequence $\gamma_n(p_n,*)\in D^2$ obtained from a minimizing sequence $\{\gamma_n\}_{n=1}^\infty\in\Pi_\Sigma$ with 
\begin{equation}
\mbox{Area}(\gamma_n(p_n,*))\to w(\Pi_\Sigma)
\end{equation}
is called a \emph{min-max sequence}.

\subsection{Non-triviality of conformal family}
Using a degree argument, we show that for M\"obius bands the topological non-triviality expressed in Proposition \ref{limits}, item (3) implies that the canonical family is \emph{geometrically} non-trivial in the following sense:
\begin{prop}\label{nontrivial}
Let $\Sigma$ be an immersed M\"obius band in $\mathbb{S}^3$ spanning $C$ with $e(\Sigma)\neq 0$.  Assume further that the density of $|\Sigma|$ is equal to $\frac{1}{2}$ on $C$.  Then
\begin{equation}
w(\Pi_\Sigma)>2\pi.
\end{equation}
\end{prop}
\begin{proof}
Note that \begin{equation} w(\Pi_\Sigma)\geq2\pi.\end{equation}  Otherwise, $C$ can be spanned by a M\"obius band with area less than $2\pi$, contradicting Lemma \ref{nonorientable}.  

By a simple compactness argument, for all $\delta>0$ there exists $\eta=\eta(\delta)>0$ so that the following holds:
\begin{equation}\label{cont}
\mbox{If } \mathbf{F}(|H_v|,|H_{w}|)\leq \eta \mbox{ then }  \mbox{dist}_{\mathbb{S}^1}(v,w)\leq \delta,
\end{equation}
Here we use the standard metric on $\mathbb{S}^1$:
\begin{equation}
\mbox{dist}_{\mathbb{S}^1}(\theta_1,\theta_2)=\min(|\theta_1-\theta_2|,2\pi-|\theta_1-\theta_2|), 
\end{equation}
for any $\theta_1,\theta_2\in [0,2\pi]$.
We now set
\begin{equation}\label{choice}
\delta=\frac{\pi}{4}\mbox{ and } \eta=\eta({\frac{\pi}{4}}   ). 
\end{equation}

Let us assume toward a contradiction that
\begin{equation}
w(\Pi_\Sigma)=2\pi.
\end{equation}
Thus there exists a sequence of sweepouts $\{u^i_p\}_{p\in\mbox{int}(D^2)}\in \Pi_\Sigma$ so that 
\begin{equation}
\sup_{p\in D^2} \mbox{Area}(u^i_p(\Sigma))\leq 2\pi +\frac{1}{i}.  
\end{equation}
By the continuity of the extended conformal family of varifolds $\{|u^i_p(\Sigma)|\}_{p\in D^2}$ (item (4) in Proposition \ref{limits}), choosing $i$ large enough, we get
\begin{equation}\label{allareclose}
\mathbf{F}(|u^i_p(\Sigma)|,\mathcal{H})\leq\frac{\eta}{3} \mbox{ for all }p\in D^2.
\end{equation}
Let $\Delta$ be a triangulation of $D^2$ and for each $k=0,1,2$ denote by $\Delta^k$ its $k$-skeleton.  Choose the triangulation fine enough so that for any two vertices $p,q\in\Delta^0$ that share an edge, we have
\begin{equation}\label{fineness}
\mathbf{F}(|u^i_p(\Sigma)|, |u^i_q(\Sigma)|)\leq\frac{\eta}{3}.
\end{equation}
By item (3) in Proposition \ref{limits} we have that $F_\Sigma=\phi\circ E_\Sigma:\partial D^2\to\mathbb{S}^1$ is continuous.  Thus we can also assume the triangulation is chosen to be fine enough that for any boundary edge $e\in\Delta^1$:
\begin{equation}\label{bondaryassumption}
\mbox{dist}_{\mathbb{S}^1}(F_\Sigma(p),F_\Sigma(q))\leq\frac{\pi}{4}\mbox{ if } \;p,q\in\mbox{supp}(e).
\end{equation}

We will now construct a continuous map 
\begin{equation}
G:D^2\to \mathbb{S}^1.
\end{equation}
so that
\begin{equation}\label{bdry}
G|_{\partial D^2}=F_\Sigma.
\end{equation}
To each interior vertex $v\in \Delta^0\setminus \partial D^2$, define $G(v)\in\mathbb{S}^1$ to by any such element satisfying
\begin{equation}\label{bounds2}
\mathbf{F}(|u^i_v(\Sigma)|,|H_{G(v)}|)\leq \frac{\eta}{3}. 
\end{equation}
This is possible by \eqref{allareclose}.  For a boundary vertex $v\in\Delta^0\cap\partial D^2$, set
\begin{equation}
G(v)=F_\Sigma(v).
\end{equation}

Let us new define $G$ over the $1$-skeleton.  For $p$ in the support of a boundary edge $e\in\Delta^1\cap\partial D^2$, set $G(p)=F_\Sigma(p)$. For an edge $e\in\Delta^1$ at least one of whose vertices is not contained in $\partial D^2$, denote by $v$ and $w$ its boundary vertices. Then by the triangle inequality, \eqref{bounds2} and \eqref{fineness}, we obtain: 
\begin{equation} \begin{aligned} \mathbf{F}\bigl(|H_{G(v)}|,|H_{G(w)}|\bigr) &\leq \mathbf{F}\bigl(|H_{G(v)}|,|u^i_v(\Sigma)|\bigr) +\mathbf{F}\bigl(|u^i_v(\Sigma)|,|u^i_w(\Sigma)|\bigr) \\ &\quad +\mathbf{F}\bigl(|H_{G(w)}|,|u^i_w(\Sigma)|\bigr) \\ &\leq \eta . \end{aligned} \end{equation}
Thus by \eqref{cont} and \eqref{choice} we have
\begin{equation}
\mbox{dist}_{\mathbb{S}^1}(G(v),G(w))\leq \frac{\pi}{4}.
\end{equation}
It follows that there exists a \emph{unique} choice of segment among the two in $\mathbb{S}^1\setminus \{G(v), G(w)\}$ with length less than $\pi$ joining $G(v)$ and $G(w)$.  On the edge $e$, define $G$ by using a homeomorphism between the support of the $1$-cell $e$ and this segment.  In this way, for any two points $x,y\in\mbox{supp}(e)$, we have
\begin{equation}\label{firstedge}
\mbox{dist}_{\mathbb{S}^1}(G(x),G(y))\leq \mbox{dist}_{\mathbb{S}^1}(G(v),G(w))\leq\frac{\pi}{4}.
\end{equation}

If $e\in\Delta^1\cap\partial D^2$ on the other hand with $\partial e = v-w$, by \eqref{bondaryassumption} we also have that for $x,y\in\mbox{supp}(e)$
\begin{equation}\label{secondedge}
\mbox{dist}_{\mathbb{S}^1}(G(x),G(y))=\mbox{dist}_{\mathbb{S}^1}(F_\Sigma(v),F_\Sigma(w))\leq \frac{\pi}{4}.
\end{equation}

Finally, let us extend $G$ over the $2$-skeleton $\Delta^2$.  Fix $\sigma\in\Delta^2$. Observe that by the triangle inequality and \eqref{firstedge} and \eqref{secondedge}, we get that for all $x,y\in\mbox{supp}(\partial \sigma)$ there holds
\begin{equation}
\mbox{dist}_{\mathbb{S}^1}(G(v),G(w))\leq\frac{\pi}{2}.
\end{equation}

Thus $G|_{\partial \sigma}$ is the image under $G$ of a circle and is contained in a segment $S=[a,b]$ of $\mathbb{S}^1$ of length less than $\pi$.  Since $\pi_1(S)$ is trivial, we can extend $G$ from the circle $\mbox{supp}(\partial \sigma)$ to the entire cell $\sigma$.  In this way we can extend $G$ over the $2$-skeleton.  We have thus obtained a continuous map \begin{equation}G:D^2\to\mathbb{S}^1\end{equation} so that by Proposition \ref{limits}, item (3), \eqref{bdry} and Lemma \ref{iseven} the restriction
\begin{equation}
G|_{\partial D^2}:\partial D^2\to \mathbb{S}^1
\end{equation}
has degree $2k$ for some $k\in\mathbb{Z}\setminus\{0\}$.  This is a contradiction.  
\end{proof}
\begin{rmk}\label{sameforgenus}
Note that the proof of Proposition \ref{nontrivial} applies equally to the conformal family associated a surface of any non-orientable genus with non-zero Euler number.  
\end{rmk}
\begin{rmk}
If $\Sigma$ satisfies $e(\Sigma)=0$, it is natural to ask whether $w(\Pi_\Sigma)=2\pi$.
\end{rmk}
\section{Min-max Theorem}\label{sectionminmax}

We need the following Min-Max theorem whose proof will appear in \cite{BKMinmax}.  
\begin{thm}[Min-max for M\"obius bands]\label{minmaxtheorem}
Let $v:M\to \mathbb{S}^3$ be a $C^1$ immersion spanning $C$ so that the density of $|v(M)|$ is equal to $\frac{1}{2}$ on $C$ with $\partial M$ and $C$ oriented so that $v|_{\partial M}:\partial M\to C$ has degree $1$.
Suppose $\Pi_\Sigma$ is the homotopy class of the associated canonical family where $\Sigma:=v(M)$.  Suppose
\begin{equation}w(\Pi_\Sigma)>\sup_{p\in\partial D^2}\mbox{Area}(\Sigma_p)=2\pi.\end{equation}  Then there exists a minimizing sequence $\{w_i(x)\}_{i=1}^\infty\in\Pi_\Sigma$ collection  $\{u_i\}_{i=1}^k$ of branched minimal immersions  
\begin{equation}
u_i:\Gamma_i\to \mathbb{S}^3,
\end{equation}
where each $\Gamma_i$ is either a M\"obius band and equal to $M$ or the standard unit disk $D\subset\mathbb{C}$ and $u_i(\partial \Gamma_i)\subset C$ for each $i\in\{1,...,k\}$.  Moreover, $\Gamma_i=M$ for at most one index $i\in\{1,...k\}$.  

There exists also a (possibly empty) collection $\mathcal{C}$ of minimal immersions
\begin{equation}
v_i:\mathbb{S}^2\to\mathbb{S}^3
\end{equation}
so that any min-max sequence $w_i:M\to\mathbb{S}^3$ for the area functional obtained from the minimizing sequence  $\{w_i(x)\}_{i=1}^\infty\in\Pi_\Sigma$ has the property that
\begin{equation}
\lim_{i\to\infty} |w_i(M)|=\sum_{i=1}^k|u_i(\Gamma_i)|+\sum_{v_i\in\mathcal{C}}|v_i(\mathbb{S}^2)|\mbox{ in the sense of varifolds,}
\end{equation}
Moreover,
\begin{enumerate}
\item Each integral $2$-varifold $|v_i(\mathbb{S}^2)|$ is a great sphere with multiplicity $m_i\in\mathbb{N}\setminus\{0\}$. 
\item If $\Gamma_i$ is a disk, then the integral $2$-varifold $|u_i(\Gamma_i)|$ is a hemisphere with multiplicity $n_i$ for some $n_i\in\mathbb{N}\setminus\{0\}$ and mass $2\pi n_i$.  Moreover, $|\mathrm{deg}(f|_{\partial\Gamma_i})|= n_i$ and $n_i\neq 0$.
\end{enumerate}
There also hold:
\begin{equation}
w(\Pi_\Sigma) = \sum_{i=1}^k|u_i(\Gamma_i)|+\sum_{v_i\in\mathcal{C}} 4\pi m_i.
\end{equation}
and
\begin{equation}\label{degreeintheorem}
1=\sum_{i=1}^k \mbox{deg}(f|_{\partial\Gamma_i})\mbox{     mod } 2, 
\end{equation}
where $\mbox{deg}(f|_{\partial\Gamma_i})$ denotes the integer degree of the mapping $f|_{\partial\Gamma_i}:\partial\Gamma_i\to C$.  

\end{thm}
\begin{rmk}
Note that a great sphere in $\mathbb{S}^3$ has index $1$, by Proposition \ref{IndLowBndProp} a M\"obius band with boundary $C$ has index at least $2$, and the hemispheres are stable.  Thus since the index is bounded from above by the number of parameters in the min-max construction (Marques-Neves \cite{MNIndex2} and Y. Sun \cite{IndexSun}) one expects in Theorem \ref{minmaxtheorem} that $k= 0$ if some $\Gamma_i$ is a M\"obius band and $k\leq 2$ if all $\Gamma_i$ are disks.  
\end{rmk}

We also obtain the following rigidity statement in \cite{BKMinmax}:
\begin{thm}\label{rigidity}
Let $v:M\to \mathbb{S}^3$ be an immersion with boundary $C$ so that the density of $|v(M)|$ is equal to $\frac{1}{2}$ on $C$.  Suppose the conformal family $\{v_{r,\theta}\}_{(r,\theta)\in D^2}$ satisfies:
\begin{equation}
\sup_{(r,\theta)\in D^2} \mbox{Area}(v_{r,\theta}(M))=m_0, 
\end{equation}
where $m_0$ is the infimal area of minimal M\"obius bands in $\mathbb{S}^3$ spanning $C$.   Then $v(M)$ is a conformal dilate of a minimal M\"obius band with area $m_0$ and boundary $C$.
\end{thm}

Obtaining harmonic mappings or minimal surfaces from a min-max process goes back to the seminal work of Sacks-Uhlenbeck \cite{SU} in 1981. Many authors have contributed to developments of the theory (see \cite{CM1} for references).  In the 1980s, J. Jost \cite{Jost} and Gulliver-Jost developed a version of Theorem \ref{minmaxtheorem} in the context of harmonic maps using successive harmonic replacements in balls of small $C^0$-oscillation.  

In 2007, T.H. Colding and W. P. Minicozzi \cite{CM1} obtained a min-max result for $2$-spheres using instead harmonic replacement on small energy balls. X. Zhou \cite{Zhoutorus} extended the result to tori and then later in \cite{Zhougenera} to all genera.  Unlike in the prior work of Jost and Colding-Minicozzi, a key point in Zhou's work was to allow the conformal structure to vary in order to obtain minimal surfaces and not merely harmonic mappings.   L. Lin,  A. Sun and X. Zhou \cite{LSZ} further extended the analysis to obtain free boundary minimal disks with boundary constrained to lie in a submanifold of arbitrary positive codimension.  

In the setting of Theorem \ref{minmaxtheorem}, one can Schwarz reflect over the geodesic to obtain closed surfaces, and Theorem \ref{minmaxtheorem} follows largely from these prior works.  

Let us record the following direct consequence of Theorem \ref{minmaxtheorem} asserting that under a width bound of $6\pi$, there is not enough energy for the maps to degenerate to a disk or have a spherical bubble:
\begin{cor}\label{existencecor}
Let $\Sigma$ be an immersed M\"obius band in $\mathbb{S}^3$ with boundary a great circle $C$ so that $|\Sigma|$ has density $\frac{1}{2}$ on $C$.  Suppose \begin{equation}\label{assume}2\pi<w(\Pi_\Sigma)<6\pi.\end{equation}  Then there exists a minimally immersed M\"obius band in $\mathbb{S}^3$ with boundary $C$, area equal to $w(\Pi_\Sigma)$, and with density $\frac{1}{2}$ on the support of $C$.
\end{cor}
\begin{proof}
 By Theorem \ref{minmaxtheorem}, the degrees of minimal immersions produced sum to an odd number, in particular there is at least one minimal immersion bounded by $C$.  That is, the surfaces $u_i(\Gamma_i)$ is a non-empty collection of disks with corresponding varifold $|u_i(\Gamma_i)|$ having mass $2\pi n_i$ or else exactly one of the $u_i(\Gamma_i)$ is a M\"obius band. 

In the first case, since 
\begin{equation}
1=\sum_{i=1}^k \mbox{deg}(f|_{\partial\Gamma_i}) \mod 2, 
\end{equation}
and $|\deg(f|_{\partial\Gamma_i})|=n_i$ for each $i\in\{1,...k\}$, we get
\begin{equation}\label{sumdegree}
1=\sum_{i=1}^k |n_i| \mod 2, 
\end{equation}
On the other hand, the assumed width bound \eqref{assume} gives that 
\begin{equation}
\sum_{i=1}^k|n_i|<3, 
\end{equation}
from which we deduce $k=1$ and $|n_i|=1$ or $k=2$ and $|n_1|=|n_2|=1$.  The latter case is ruled out by \eqref{sumdegree}.   But if $k=1$ and $n_1=1$, then \eqref{assume} gives that
\begin{equation}\label{wbound}
2\pi<2\pi+\sum_{v_i\in\mathcal{C}}4\pi m_i<6\pi, 
\end{equation}
where $m_i$ is the multiplicity of the associated minimal sphere.   If $\mathcal{C}$ is empty \eqref{wbound} is impossible as it violates the lower bound, while if $\mathcal{C}$ is non-empty, it is also impossible as it violates the upper bound.  

Thus we can assume exactly one of the minimal surfaces $\{\Gamma_i\}_{i=1}^k$, say $\Gamma_1$, is a M\"obius band and the remaining minimal surfaces $\{\Gamma_2,...,\Gamma_k\}$ are disks. By Lemma \ref{nonorientable},
\begin{equation}\label{bbb}
\mbox{Area}(\Gamma_1)\geq 4\pi.
\end{equation}
Since each disk $\{\Gamma_2,...\Gamma_k\}$ has area at least $2\pi$, from the width bound \eqref{assume} and \eqref{bbb} it follows that this collection is empty and $k=1$.  Moreover, this also implies the collection of minimal $2$-spheres $\mathcal{C}$ is empty.

Thus from \eqref{degreeintheorem} we get that $\mbox{deg}(u_1|_{\partial M})=1$, as desired.  By the upper density bound of $6\pi$ \eqref{assume}, the minimal surface $u_1(M)$ has density equal to $\frac{1}{2}$ on the support of $C$.

\end{proof}
\section{Least area minimal M\"obius bands}\label{areaboundssection}
In this section we consider minimal M\"obius bands of least possible area bounded by $C$.
Let $M$ be an abstract M\"{o}bius band and  let
\begin{equation*}
\mathcal{M}=\{u\;|\; u:M\to\mathbb{S}^3\mbox{ is a branched minimal immersion spanning } C\},
\end{equation*}
this is a non-empty set as it contains $u_{LB}$ a parameterization of the Lawson band.
Set 
\begin{equation}
m_0:=\inf_{u\in\mathcal{M}}\mbox{Area}(u)\leq Area(u_{LB})<6\pi.
\end{equation}
Observe that  Lemma \ref{nonorientable} item (1) immediately implies
$$
m_0\geq 4\pi.
$$
This infimum is achieved:
\begin{prop}
There exists $u\in\mathcal{M}$ so that setting $\Sigma:=u(M)$, there holds
\begin{equation}
 \mbox{Area}(\Sigma)=m_0\in [4\pi, 6\pi).
\end{equation}
Morever, we may choose $u$ to be a smooth embedding.
\end{prop}
\begin{proof}
Let $\Sigma_i\in\mathcal{M}$ be a minimizing sequence with $\mbox{Area}(\Sigma_i)\to m_0$.  Then for a fixed $\gamma>0$ small and after possibly throwing out a finite number of $i$ we have, 
\begin{equation}
\mbox{Area}(\Sigma_i)\leq m_0+\gamma< 6\pi.
\end{equation}
By Proposition \ref{NoBPProp} all the $\Sigma_i$ are smoothly embedded minimal M\"{o}bius bands spanning $C$.

It will be convenient to consider $\hat{\Sigma}_i$ the immersed minimal surfaces obtained by Schwarz reflecting the $\Sigma_i$ across $C$.  One has
$$
\mbox{Area}(\hat{\Sigma}_i)=2 \mbox{Area}(\Sigma_i).
$$
As $\hat{\Sigma}_i$ is a smooth minimal immersion of a Klein bottle, it cannot be embedded.  All such surfaces satisfy $\mathrm{Area}(\hat{\Sigma}_i)=2m_0+2\gamma<12\pi$. 

By \cite[Theorem 3]{WhiteCpct} and standard elliptic theory, there is a finite set $S\subset \mathbb{S}^3$ so that, up to passing to a subsequence,
$$
\hat{\Sigma}_i\to \hat{\Sigma}_\infty
$$
in $C^\infty_{loc}(\mathbb{S}^3 \setminus S)$
where $\hat{\Sigma}_\infty$ is a surface stationary for area.  In particular, it is a smoothly immersed minimal surface away from $S$.   In fact it is a branched minimal surface, possibly with branch points at the points of $S$. 

We claim $S=\emptyset$.  To that end consider the sequence of Radon measures determined by
$$
\mu_i(U)=\frac{1}{2}\int_{\hat{\Sigma}_i\cap U} |A_{\hat{\Sigma}_i}|^2. 
$$
By the Gauss-Bonnet theorem and minimality of $\hat{\Sigma}_i$ in $\mathbb{S}^3$, as $\chi(\hat{\Sigma}_i)=0$,
\begin{equation}
\frac{1}{2}\int_{\hat{\Sigma}_i} |A_{\hat{\Sigma}_i}|^2=\mbox{Area}(\hat{\Sigma}_i)-2\pi\chi(\hat{\Sigma_i})=\mbox{Area}(\hat{\Sigma}_i).
\end{equation}
Thus by the assumed area bound we get
$$
\mu_i(\mathbb{S}^3)<  12\pi \mbox{ and } \limsup_{i\to \infty} \mu_i(\mathbb{S}^3)=2m_0. 
$$
Up to passing to a subsequence, $\mu_i\to \mu_\infty$ in the weak* topology and
$$
\mu_{\infty}(\mathbb{S}^3)=2m_0 <12\pi.
$$
We note that Meeks \cite{Meeks} showed that any complete non-orientable minimal surface has to have total curvature at least $6\pi$.  Hence, the conclusions of  \cite{WhiteCpct} can be generalized to show that either
$$
\mu_\infty(\set{p})=0, p\not \in S \mbox{ or } \mu_\infty(\set{p})\geq 2\pi k, p\in S, k\geq 2.
$$
Hence,  $|S|\leq 2$. If $S=\emptyset$, then we are done, so we may suppose $|S|\geq 1$.

We further observe that by combining \cite{WhiteCpct},  \cite{SchoenUniq}  and a standard point picking argument (e.g., \cite[Theorem 2.2]{CMBook}), after possibly passing to a further subsequence the following is true: For each $p\in S$,  there are $p_i \to p $ and $\rho_i\to \infty$ so that $\rho_i (\hat{\Sigma}_i-p_i)\to \Gamma$ where here the translation and dilation take place in $\mathbb{R}^4$, the convergence is in the locally smooth sense (possibly with multiplicity) and $\Gamma$ is, after rotating appropriately in $\mathbb{R}^4$, either:
\begin{enumerate}
	\item A catenoid in $\mathbb{R}^3$ with multiplicity one;
	\item A collection of minimal surfaces in $\mathbb{R}^3$, at least one of which is not flat, with $\lim_{r\to \infty} \frac{1}{\pi r^2} \mathrm{Area}(\Gamma\cap B_{r}(q))\geq 3$.
\end{enumerate} 
Equivalently, for any $\epsilon>0$ there is a $\Lambda\geq 1$ so that, for $r_i=\Lambda \rho_i^{-1}\to 0$, either
\begin{enumerate}
	\item  $A_i^\epsilon=B_{r_i}(p_i)\cap \hat{\Sigma}_i$ is an annulus and $\liminf_{i\to \infty}\int_{A_i^\epsilon} |A_{\hat{\Sigma}_i}|^2 \geq 4\pi-\epsilon$ and $\rho_i (A_i-p_i)$ converges to a compact subset of a (appropriately scaled) catenoid;
	\item $\liminf_{i\to \infty}\mbox{Area}(B_{r_i}(p_i)\cap\hat{\Sigma}_i) \geq (3-\epsilon)\pi r_i^2$.
\end{enumerate}
While not needed in the proof we note that a complete classification of the possible minimal surfaces obtained in Case (2) is given in \cite{Lopez}.

We first show that Case (2) cannot occur.  Indeed, suppose for contradiction that  Case (2) holds for $p\in S$.  By the monotonicity formula (\cite{SimomBook}), for any $\delta>0$ there is $r_\delta>0$ small so if $r_\delta>r_2>r_1>0$ and $\Sigma$ is a minimal surface in $\mathbb{S}^3$ and $q$ is some point, then
$$
(1+\delta)\frac{1}{\pi r_2^2} \mbox{Area}(B_{r_2}(q)\cap \Sigma)  \geq  \frac{1}{\pi r_1^2} \mbox{Area}(B_{r_1}(q)\cap \Sigma).
$$
It follows that for fixed $\delta>0$ small, $i$ sufficiently large and any $r_\delta>r>r_i$
$$
(1+\delta)\frac{1}{\pi r^2} \mbox{Area}(B_{r}(p_i)\cap \hat{\Sigma}_i) + \delta \geq 3. 
$$
Letting $i\to \infty$ we obtain, for all $0<r<r_\delta$, that
$$
(1+\delta)\frac{1}{\pi r^2} \mbox{Area}(B_{r}(p)\cap \hat{\Sigma}_\infty) + \delta \geq 3. 
$$
It follows that
$$
\lim_{r\to 0^+} \frac{1}{\pi r^2} \mbox{Area}(B_{r}(p)\cap \hat{\Sigma}_\infty)\geq 3.
$$
This implies, either via conformal blow up or tangent cone analysis that 
$$
\mbox{Area}(\hat{\Sigma}_\infty)\geq 12 \pi,
$$
which is a contradiction.  That is, Case (2) does not occur.

Hence, for each $p\in S$ the situation is as described in Case (1).  Let us fix an $\epsilon>0$ small.  For $i$ is sufficiently large we may take $A_i^\epsilon\subset \Sigma_i$.  This is immediate when $p\not\in C$.  When $p\in C$,  the translation and rescaling of $C$ converges to a straight line.  However,  a catenoid does not contain such a line and so, for $i$ large enough, $A_i^\epsilon\cap C=\emptyset$.  Since $C$ separates $\hat{\Sigma}_i$ into $\Sigma_i$ and its reflection across $C$, it follows that, up to reflecting across $C$, $A_i^\epsilon\subset \Sigma_i$.

It follows that $\gamma_i^{+}\cup \gamma_i^{-}=\partial A_i^\epsilon$ are two curves $\hat{\Sigma}_i$ that correspond to large circles in the limit catenoid.  As they are disjoint curves that are homotopic, and $\Sigma_i$ is a M\"{o}bius band, it follows that both $\gamma_i^{\epsilon,\pm}$ separate $\Sigma_i$. Hence, there is a $D_i\subset \Sigma_i$ so, up to relabelling $\gamma_i^{-}\subset D_i$ and $\partial D_i=\gamma_i^{,+}$ where here $D_i$ is either a disk or a M\"obius band.  There is also a copy of $D_i$ obtained by reflecting across $C$ which we denote by $D_i^{*}$.  We note that, by construction, $C\subset \hat{\Sigma}_i'=\hat{\Sigma}_i\setminus (D_i\cup D_i^{*})$.

We note that, as $D_i$ is not contained inside $B_{r_i}(p_i)$, it follows from the maximum principle that $D_{i}$ is not contained inside $B_{\pi}(p_i)$.  The same is true of $D_{i}^{*}$.  It follows that $D_i$ converge to a non-trivial minimal surface, $D_{\infty}$ in $\mathbb{S}^3\setminus \{p\}$.  The limit is stationary in all of $\mathbb{S}^3$ and so has area at least $4\pi$.  The same is true of $D_{i}^*$.  Finally, as $\hat{\Sigma}_i'$ contains $C$, it also converges to a non-trivial stationary surface which has area at least $4\pi$.  That is 
$$
\liminf_{i\to \infty}\mbox{Area}(\hat{\Sigma}_i) \geq \liminf_{i\to \infty}\mbox{Area}(\hat{\Sigma}_i')+ 2\liminf_{i\to \infty}\mbox{Area}(D_i)\geq 12\pi.
$$
This is a contradiction and implies $\hat{\Sigma}_\infty$ is a smoothly immersed surface and the $\hat{\Sigma}_i$ converge smoothly to the limit in the appropriate sense.  

It follows that the $\Sigma_i$ converge smoothly to a smooth minimal immersion of a M\"{o}bius band $\Sigma_\infty$ spanning $C$ and with area $m_0<6\pi$.  By Proposition \ref{NoBPProp}, $\Sigma_\infty$ is smoothly embedded.

\end{proof}

\subsection{Conformal volume}
Suppose $\Sigma$ is a subset of $\mathbb{S}^3$.  Recall the notion of \emph{conformal volume} $\mathcal{V}_c(\Sigma)$ introduced by Li-Yau:
\begin{equation}\label{liyau}
\mathcal{V}_c(\Sigma)=\sup_{v\in B^4}\mbox{Area}(F_v(\Sigma)).
\end{equation}

Observe from the definition \eqref{liyau} that if $\Sigma\subset \mathbb{S}^3$ is a $C^2$ surface spanning the geodesic $C$ and $\Sigma'=\Sigma\cup\tau(\Sigma)$, where $\tau$ is the Schwarz reflection about $\Sigma$, then 
\begin{equation}\label{double}
2\mathcal{V}_c(\Sigma)\geq \mathcal{V}_c(\Sigma')\geq\mathcal{V}_c(\Sigma).
\end{equation}
The surface $\Sigma'$ is in general a $C^{1,1}$-immersed surface.  If $\Sigma$ is minimal, then $\Sigma'$ is smooth.

When $\Sigma$ and thus $\Sigma'$ are minimal, then a result of Li-Yau is that
\begin{equation}\label{liyauminimal}
\mathcal{V}_c(\Sigma')=\mbox{Area}(\Sigma')=2\mbox{Area}(\Sigma).
\end{equation}

We can now obtain lower bounds for the conformal volume of M\"obius bands with Euler number zero:

\begin{prop}\label{eulerzero}
Let $\Sigma$ be an immersed M\"obius band spanning $C$ with $e(\Sigma)=0$.  Suppose the density of the integral $2$-varifold $|\Sigma|$ on $C$ is equal to $\frac{1}{2}$.  Then 
\begin{equation}
\mathcal{V}_c(\Sigma)\geq 12\pi.  
\end{equation}
In particular, if $\Sigma$ is in addition minimally immersed\footnote{We show in \cite{BernKetProperties} that there is no such minimal immersion.  This generalizes a result of Almgren asserting that $\mathbb{RP}^2$ itself does not minimally immerse in $\mathbb{S}^3$.} then \begin{equation}\mbox{Area}(\Sigma)\geq 6\pi.\end{equation}
\end{prop}
\begin{proof}

It follows from the density $\frac{1}{2}$ condition on $C$ and $e(\Sigma)=0$ that (letting $N$ denote a small enough closed tubular neighborhood of $C$ whose boundary is transverse to $\Sigma$) the surface $\Sigma\cap N$ consists of a single annular ``ribbon" intersecting the torus $\partial N$ in a single simple closed curve going around once in the longitudinal direction and zero times in the meridian direction.  

We consider in the complementary solid Heegaard torus $T=\overline{\mathbb{S}^3\setminus N}$ the M\"obius band $M=\Sigma\cap T$.  The band $M$ intersects $\partial T$ in a single simple closed curve $\gamma$ which by the previous paragraph bounds a meridian disk in $T$.  We claim that $M$ has a triple point in $T$.  Let us assume toward a contradiction that it does not.

Consider the inclusion $\iota:M\to T$.   Recall $\pi_1(M)\cong\mathbb{Z}$ and $\pi_1(T)\cong\mathbb{Z}$.  Let us denote $\iota_*:\pi_1(M)\to\pi_1(T)$.  If $g\in \pi_1(M)$ is a generator and $b\in \pi_1(M)$ is the element corresponding to a paramterization of $\gamma=\partial M$ then up to reversing the orientation of the parameterization $b=2g$.  The fact that $\partial M$ bounds a disk in $T$ ensures 
\begin{equation}\iota_*(2g)=\iota_*(b) = 0\in\pi_1(T).\end{equation}  
Thus \begin{equation}2\iota_*(g)=0\in \pi_1(T)\cong\mathbb{Z}.\end{equation}
It follows that $\iota_*(g)=0$ which implies \begin{equation}\label{lifting}\iota_*(\pi_1(M))=0.\end{equation}  Let $\tilde{T}$ denote the universal cover of $T$ and let $\pi:\tilde{T}\to T$ be a covering map.  The manifold $\tilde{T}$ is homeomorphic to $D^2\times\mathbb{R}$.  Then by the lifting criterion, \eqref{lifting} implies that $\iota:M\to T$ lifts to a map $\tilde{\iota}:M\to \tilde{T}$ with $\pi\circ\tilde{\iota}=\iota$.  Thus $\tilde{M}:=\tilde{\iota}(M)\subset \tilde{T}$ is compact and has one boundary component which is isotopic to $\partial D^2\times \{a\}$ for any $a\in\mathbb{R}$.  Since the lift of a surface under a covering map has at most as many triple points as the original surface, it follows that $\tilde{M}$ has no triple points in $D^2\times\mathbb{R}$. 

Since the boundary of $\tilde{M}$ is isotopic to that of a meridian disk we can find $H\in\mathbb{R}$ so that \begin{equation}\tilde{M}\cap (D^2\times\{\pm H\})=\emptyset\end{equation} and \begin{equation}\tilde{M}\subset D^2\times [-H,H].\end{equation}  The surface $A\subset \partial D^2\times \mathbb{R}$ bounded between $\partial D^2\times \{H\}$ and $\gamma\subset \partial D^2\times\mathbb{R}$ is an embedded annulus. Adjoining to $\tilde{M}$ the annulus $A$ and the disk $D^2\times\{H\}$ (smoothing an arbitrarily small amount at the corners) we obtain a closed manifold diffeomorphic to $\mathbb{RP}^2$ that is smoothly immersed in the three-ball $D^2\times [-H,H]$ with no triple points.  This contradicts T. Banchoff's theorem \cite{Banchoff} which implies that any smoothly immersed $\mathbb{RP}^2$ in $\mathbb{R}^3$ has an odd number of triple points.  Thus the claim is established and $M$ has at least one triple point $p\in M$.  

By considering conformal dilations based at $p$ we obtain
\begin{equation}
\mathcal{V}_c(\Sigma)\geq 12\pi.
\end{equation} 
Thus by \eqref{double} we get 
\begin{equation}
\mathcal{V}_c(\Sigma')\geq 12\pi.
\end{equation}
By \eqref{liyauminimal}  we obtain when $\Sigma'$ is minimal
\begin{equation}
\mbox{Area}(\Sigma)\geq6\pi.
\end{equation}
\end{proof}
We also need the following fact about conformal volume asserting that maximal volume under conformal dilates is uniquely achieved at the identity.  
\begin{prop}\label{soufi}
Suppose $\Sigma$ is the image of a branched minimal immersion has boundary $C$ and $|\Sigma|$ has density $\frac{1}{2}$ on $C$.  Then if 
\begin{equation}\label{equality}
\sup_{(x,y)\in D^2}\mbox{Area}(\Sigma_{x,y})\leq \mbox{Area}(\Sigma).
\end{equation}
Moreover, if $\mbox{Area}(\Sigma_{x,y})=\mbox{Area}(\Sigma)$ for some $(x,y)\in D^2$ and $(x,y)\neq (0,0)$, then $\Sigma$ is a hemisphere.  
For $x^2+y^2=1$ and $t$ small there holds
\begin{equation}\label{confd}
\mbox{Area}(\Sigma_{tx,ty})=\mbox{Area}(\Sigma)-2t^2\int_{\Sigma'} (xn_1+yn_2)^2dA +O(t^3), 
\end{equation}
where $n_i$ is the component of a choice of normal $\mathbf{n}_{\Sigma'}$ on the orientation double cover $\Sigma'$ of $\Sigma$ in the $\mathbf{e}_i$ direction in $\mathbb{R}^4$.
\end{prop}
\begin{proof}
Li-Yau (equation 3.4 in \cite{LY}) proved that if $\Sigma$ is minimal then for any any conformal map $g\in\mbox{Conf}(\mathbb{S}^3)$, we have
\begin{equation}\label{liyauargument}
\mbox{Area}(\Sigma)=\frac{1}{4}\int_{g(\Sigma)}|H|^2dA + \mbox{Area}(g(\Sigma)). 
\end{equation}
This gives \eqref{equality}.  For the equality case, suppose here is a pair $(x,y)\in D$  so that 
\begin{equation}\label{AreaSigxySig}
\mbox{Area}(\Sigma_{x,y})=\mbox{Area}(\Sigma).
\end{equation}
If $(x,y)\in\mbox{int}(D^2)$, then a result of A. El Soufi and S. Ilias (Theorem 1.1, \cite{ElSoufi}) implies $\Sigma$ is a hemisphere.  If $(x,y)\in \partial D$, then  the density $\frac{1}{2}$ condition and \eqref{AreaSigxySig} implies $\mbox{Area}(\Sigma)=2\pi$ and so, by  Lemma \ref{orientable} item (1), $\Sigma$ is a hemisphere. 

To see \eqref{confd}, observe that the derivative of the one-parameter family of diffeomorphisms $\{F_{tv}\}_{t\in(-\varepsilon,\varepsilon)}$ at $t=0$ is given by:
\begin{equation}\label{variation}
X_{v}(x)=2(\langle x,v\rangle x-v).
\end{equation}
Let us show \eqref{variation}.
Fix \(v\in \mathbb{S}^3\subset \mathbb{R}^4\), with \(|v|=1\), and let \(x\in \mathbb{S}^3\). Consider the one-parameter family
\begin{equation}\label{defn}
F_{t v}(x)
=
\frac{1-t^2}{|x-tv|^2}(x-tv)-tv .
\end{equation}

Since \(|x|=|v|=1\), we have
\begin{equation}
|x-tv|^2
=
|x|^2-2t\langle x,v\rangle+t^2|v|^2
=
1-2t\langle x,v\rangle+t^2 .
\end{equation}

Set
\begin{equation}
s=\langle x,v\rangle .
\end{equation}
Then expanding \eqref{defn} we get
\begin{equation}
F_{t v}(x) = \frac{1-t^2}{1-2ts+t^2}(x-tv)-tv .
\end{equation}
Define\label{at}
\begin{equation}
A(t)
=
\frac{1-t^2}{1-2ts+t^2}.
\end{equation}
We see that
\begin{equation}
A(0)=1.
\end{equation}
Thus let us rewrite:
\begin{equation}
F_{t v}(x)
=
A(t)(x-tv)-tv, 
\end{equation}
\noindent
Differentiating \eqref{at}, we get that
\begin{equation}
A'(0)=2s=2\langle x,v\rangle .
\end{equation}
Thus differentiating \(F_{t v}(x)\) at \(t=0\), we obtain
\begin{equation}
\begin{aligned}
\left.\frac{d}{dt}\right|_{t=0}F_{t v}(x)
&=
A'(0)x+A(0)(-v)-v \\
&=
2\langle x,v\rangle x-v-v \\
&=
2\langle x,v\rangle x-2v, 
\end{aligned}
\end{equation}
as desired. Note that the vector field $X_v(x)$ is indeed tangent to $\mathbb{S}^3$ since for any $x\in\mathbb{S}^3$ we have $|x|^2=1$ and thus
\begin{equation}
\langle X_v(x),x\rangle = 2(\langle x, v\rangle\langle x, x\rangle -\langle v, x\rangle)=0.
\end{equation}

If $\Sigma'$ is the orientation double cover of $\Sigma$ and $\mathbf{n}_{\Sigma'}$ is a choice of unit normal along $\Sigma'$, then we consider the normal part of the vector field $X_{v(x,y)}$:
\begin{equation}\label{normalpart}
\phi_{v}(z)=\langle X_{v}(z),\mathbf{n}_{\Sigma'}(z)\rangle = 2(\langle z,v\rangle z-v), \mathbf{n}_{\Sigma'}(z)\rangle =-2\langle v,\mathbf{n}_{\Sigma'}(z)\rangle, 
\end{equation}
where the last equality follows since $\langle z, \mathbf{n}_{\Sigma'}(z)\rangle=0$ for $z\in\mathbb{S}^3$.  Now set \begin{equation}v(x,y)=\langle x, y, 0,0\rangle.\end{equation}  Thus we get from \eqref{normalpart}
\begin{equation}
\phi_{v(x,y)}(z)=-2(xn_1(z)+yn_2(z)), 
\end{equation}
where $n_i$ is the projection of $\mathbf{n}_{\Sigma'}$ on the $i$th coordinate direction in $\mathbb{R}^4$.
As observed in Section \ref{indexsection} (\eqref{eigenvalue} and \eqref{dirichlet}), $\phi_{v(x,y)}(z)$ is an odd Dirichlet eigenfunction for $L_{\Sigma'}$ with eigenvalue $-2$ which vanishes on $C$.  The expansion \eqref{confd} follows from the standard Taylor expansion of the area functional since the linear term vanishes by minimality and the second derivative is given by 
\begin{equation}
\int_{\Sigma'}\phi_{v(x,y)}L\phi_{v(x,y)} dA = -8\int_{\Sigma'}(xn_1+yn_2)^2dA <0.
\end{equation}
\end{proof}
We now prove the main result of this section:
\begin{thm}\label{leastareamb}
Any $\Sigma\in\mathcal{M}$ with $\mbox{Area}(\Sigma)=m_0$ is smoothly embedded, has $\mbox{ind}_D(\Sigma)=2$ and $e(\Sigma)=\pm 2$.
\end{thm}
\begin{proof}
Let $u\in\mathcal{M}$ have $\Sigma:=u(M)$ with area equal to $m_0$. Without loss of generality we can assume 
\begin{equation}\label{areaass}
\mbox{Area}(\Sigma)\leq\mbox{Area}(\bar{\tau}_{1,2})<6\pi.
\end{equation}
Hence, by Proposition \ref{NoBPProp}, $\Sigma$ is smoothly embedded.

If $\mbox{ind}_D(\Sigma)>2$, by Proposition \ref{soufi} and the analogous argument to Proposition 3.1 in \cite{MNRigidity} we can find an element in $\{u'_{x,y}\}_{x\in D^2}\in\Pi_\Sigma$ with 
\begin{equation}\label{lower}
\max_{(x,y)\in D^2}\mbox{Area}(u'_{x,y})< \mbox{Area}(\Sigma). 
\end{equation}
Let us give the details. If $\mbox{ind}_D(\Sigma)>2$, then there exists an odd negative eigenfunction $\phi_3$ for $L_{\Sigma'}$ with negative eigenvalue $-\lambda_3$ on the orientation double cover $\Sigma'$ of $\Sigma$ that is orthogonal to the space of odd eigenfunctions spanned by $\{n_1,n_2\}$, i.e. the $\mathbf{e}_1$ and $\mathbf{e}_2$ components of the unit normal $\mathbf{n}_{\Sigma'}$ to $\Sigma'$.  Let us consider the vector field $V=\phi_3\mathbf{n}_{\Sigma'}$ on $\Sigma'$.  Extend it to be supported on a neighborhood of $\Sigma'$ and let $\{D_t\}_{t\in [-\varepsilon,\varepsilon]}$ denote the one-parameter family of diffeomorphisms of $\mathbb{S}^3$ generated by the flow of $V$. Let us denote for each $(x,y)\in\partial D^2$ (for $t$ and $s$ small enough):
\begin{equation}
g(t,s)=\mbox{Area}(D_t(\Sigma_{sx,sy})).
\end{equation}
Then we have $g_t(0,0)=g_s(0,0)=0$ by minimality of $\Sigma$.  We also have $g_{tt}(0,0)=-\lambda_3<0$.  Letting $\phi_{v(x,y)}=xn_1+yn_2$ we get from \eqref{confd}
\begin{equation}
 g_{ss}(0,0)=\int_{\Sigma'} \phi_{v(x,y)}L\phi_{v(x,y)}dA=-2\int_{\Sigma'} (xn_1+yn_2)^2dA<-C_{x,y}
\end{equation} 
for some $C_{x,y}>0$.  The inequality follows since otherwise, $n_1$ and $n_2$ are linearly dependent on $\Sigma'$ which is impossible by Proposition \ref{IndLowBndProp} as $\Sigma'$ is an annulus and not a hemisphere by assumption.

We also have
\begin{equation}
g_{ts}(0,0)=g_{st}(0,0)=\int_{\Sigma'} \phi_3 L\phi_{v(x,y)}dA=-2\int_{\Sigma'} \phi_3\phi_{v(x,y)}dA=0.
\end{equation}
Thus, for $t$ and $s$ small enough we can expand
\begin{equation}
\mbox{Area}(D_t(\Sigma_{sx,sy}))\leq \mbox{Area}(\Sigma)-\frac{C}{2}s^2-\lambda_3t^2+O(|s|^2+|t|^2)^{3/2}), 
\end{equation}
where $C>0$ is given as 
\begin{equation}
C=\min_{(x,y), x^2+y^2=1}C_{x,y}>0.
\end{equation}
In particular (after relabelling the variables $x$ and $y$):
\begin{equation}\label{perturb}\mbox{Area}(D_t(\Sigma_{x,y}))<\mbox{Area}(\Sigma).\end{equation} 
for $t$ near $0$ and $(x,y)\in D^2$ near $(0,0)\in D^2$.  By Proposition \ref{soufi}, the maximal area of a slice of the family $\{u_{x,y}(\Sigma)\}_{(x,y)\in D^2}$ occurs uniquely at $(x,y)=(0,0)$. Thus by continuity of area under diffeomorphisms and \eqref{perturb} we get that for $t$ small enough there holds:
\begin{equation}
\max_{(x,y)\in D^2}\mbox{Area}(D_t(u_{x,y}))<\mbox{Area}(\Sigma), 
\end{equation}
as desired.  Clearly we have $\{D_t(\Sigma_{x,y})\}_{(x,y)\in D^2}\in \Pi_\Sigma$ for small enough $t$ as shrinking $t$ down to $0$ gives an explicit homotopy and the boundary of the family is unchanged along the homotopy so that all conditions in \eqref{homotopydef} are satisfied.   Thus we can set $u'_{x,y}=D_{t_0}(u_{x,y})$ for some choice of $t_0>0$ so that \eqref{lower} holds.  

From this and the definition of $w(\Pi_\Sigma)$ we  obtain
\begin{equation}
w(\Pi_\Sigma)<\mbox{Area}(\Sigma)\leq \mbox{Area}(\bar{\tau}_{1,2})<6\pi.
\end{equation}
By Proposition \ref{eulerzero}, we can assume without loss of generality that $e(\Sigma)\neq 0$. By Proposition \ref{nontrivial}, we then infer
\begin{equation}
w(\Pi_\Sigma)>2\pi.
\end{equation}
By Corollary \ref{existencecor}, it follows that there exists a minimal M\"obius band $\Sigma'$ with boundary $C$ satisfying
\begin{equation}\label{lowest}
4\pi\leq w(\Pi_\Sigma)=\mbox{Area}(\Sigma')<\mbox{Area}(\Sigma)<6\pi.
\end{equation}
The equation \eqref{lowest} is a contradiction to the fact that $\Sigma$ had least area among such minimal M\"obius bands.  Thus, the index of $\Sigma$ was in fact at most $2$ and by Proposition \ref{IndLowBndProp} the index of $\Sigma$ is equal to $2$. As $\Sigma$ is embedded a result of Whitney \cite{Whitney} and Massey \cite{Ma} -- see also \cite[Corollary 1.1.1]{Yasuhara} --  implies $e(\Sigma)=\pm 2$.
 \end{proof}

\section{Conformal area and Willmore energy}\label{applicationsection}

When $\Sigma$ has boundary $C$, let us define the conformal volume restricted to conformal maps preserving $C$:
\begin{equation}
\mathcal{V}_c^C(\Sigma)=\sup_{(r,\theta)\in D^2}\mbox{Area}(\Sigma_{r,\theta}). 
\end{equation}
Let us denote $\Sigma'=\Sigma\cup\tau(\Sigma)$, where $\tau$ is Schwarz-reflection through $C$.  Note that since for each $(r,\theta)\in D^2$ we have (as both surfaces are isometric):
\begin{equation}
\mbox{Area}(\Sigma_{r,\theta})=\mbox{Area}((\tau(\Sigma))_{r,\theta})
\end{equation}
we get
\begin{equation}\label{twicegood}
2\mathcal{V}_c^C(\Sigma)=\mathcal{V}_c^C(\Sigma').
\end{equation}
Moreover from the definition we have
\begin{equation}\label{better}
 \mathcal{V}_c(\Sigma)\geq \mathcal{V}_c^C(\Sigma)\mbox{ and }  \mathcal{V}_c(\Sigma')\geq \mathcal{V}_c^C(\Sigma').
\end{equation}
We obtain the following lower bounds for conformal volume:
\begin{thm}\label{conflowerbounds}
Let $\Sigma$ be an immersed M\"obius band with boundary a great circle.   If $e(\Sigma)=0$, then
$$
m_0<6\pi\leq \mathcal{V}_c(\Sigma).
$$
If $e(\Sigma)\neq 0$, then
\begin{equation}
m_0\leq \mathcal{V}^C_c(\Sigma)\leq \mathcal{V}_c(\Sigma), 
\end{equation}
with equality if and only if $\Sigma$ is a conformal dilate of an element of $\mathcal{M}^*$. 
\end{thm}
\begin{proof}
When $e(\Sigma)=0$ the first claim follows from the fact that $m_0<6\pi$ and by Proposition \ref{eulerzero}. 

As $m_0<6\pi$ we may, without loss of generality assume for the second claim that
\begin{equation}
\mathcal{V}_c^C(\Sigma)<6\pi.  
\end{equation}
Note that as observed above
\begin{equation}
\sup_{(\theta,t)\in D^2}\mbox{Area}(\Sigma_{r,\theta})=\mathcal{V}_c^C(\Sigma)\leq \mathcal{V}_c(\Sigma). 
\end{equation}
Thus if $\Pi_\Sigma$ is the homotopy class associated to the conformal family, then
\begin{equation}
w(\Pi_\Sigma)\leq \mathcal{V}^C_c(\Sigma)<6\pi. 
\end{equation}
As $e(\Sigma)\neq 0$, Proposition \ref{nontrivial}, implies
\begin{equation}2\pi <w(\Pi_\Sigma).\end{equation}  
Thus we obtain from Theorem \ref{minmaxtheorem} that there exists a minimal M\"obius band $\Gamma$ with boundary $C$ so that
\begin{equation}
\mbox{Area}(\Gamma)=w(\Pi_\Sigma).
\end{equation}
By Theorem \ref{leastareamb} we have
\begin{equation}
\mbox{Area}(\Gamma_0)\leq \mbox{Area}(\Gamma), 
\end{equation}
where $\Gamma_0$ is an infimal area minimal M\"obius band with boundary $C$.  

In the case of equality, we get $m_0=\mathcal{V}_c^C(\Sigma)$.  By Theorem \ref{rigidity}, $\Sigma$ is a conformal dilate of an infimal area minimal M\"obius band with boundary $C$.
\end{proof}
We obtain directly from Theorem \ref{conflowerbounds}, \eqref{twicegood} and \eqref{better} the following;
\begin{cor}\label{doublesit}
Suppose $\Sigma$ is an immersed Klein bottle in $\mathbb{S}^3$ containing the closed geodesic $C$ and invariant under Schwarz reflection through $C$.  Then 
\begin{equation}
\mathcal{V}_c(\Sigma)\geq 2m_0
\end{equation}
with equality if and only if $\Sigma$ is a conformal dilate of an element in $\mathcal{M}^*$ preserving $C$.  
\end{cor}

Recall that the Willmore energy $\mathcal{W}(\Sigma)$ for a surface $\Sigma$ with boundary is given by 
\begin{equation}\label{willmoredef}
\mathcal{W}(\Sigma)=\int_\Sigma(1+\frac{1}{4}|\mathbf{H}_\Sigma|^2)dA + \int_{\partial\Sigma}k_g dL.
\end{equation}
Thus when $\partial\Sigma$ is a geodesic, i.e., $k_g=0$, 
\begin{equation}
 \mathcal{W}(\Sigma)\geq \mathrm{Area}(\Sigma). 
\end{equation}
Moreover, as $\Sigma'$ is a $C^{1,1}$ closed surface, its Willmore energy is well defined and satisfies
\begin{equation}
\mathcal{W}(\Sigma')= \int_{\Sigma'} (1+\frac{1}{4}|\mathbf{H}_\Sigma|^2) dA= \mathcal{W}(\Sigma)+\mathcal{W}(\tau(\Sigma))= 2\mathcal{W}(\Sigma).
\end{equation}

Let us prove 

\begin{thm}
Let $\Sigma$ be a M\"obius band immersed in $\mathbb{S}^3$ with boundary $C$.  Then 
\begin{equation}\label{willmorebound}
\mathcal{W}(\Sigma)\geq m_0
\end{equation}
with equality if and only if $\Sigma$ is a conformal dilate of an element of $\mathcal{M}^*$.
\end{thm}

\begin{proof}
If $e(\Sigma)=0$, then the conformal invariance of the Willmore energy of $\Sigma'$ implies
$$
\mathcal{W}(\Sigma)=\frac{1}{2}\mathcal{W}(\Sigma')\geq \frac{1}{2} \mathcal{V}_c(\Sigma')\geq \frac{1}{2} \mathcal{V}_c(\Sigma)\geq 6\pi>m_0.
$$
This gives \eqref{willmorebound} when $e(\Sigma)=0$.  Assume next that $e(\Sigma)\neq 0$.  Then by the conformal invariance of the Willmore energy \eqref{willmoredef} with boundary we get
\begin{equation}
\mathcal{W}(\Sigma_{r,\theta})=\mathcal{W}(\Sigma)\mbox{ for all }  (r,\theta)\in D^2.
\end{equation}
 Hence,
\begin{equation}
\mathcal{V}_c^C(\Sigma)=\sup_{(r,\theta)\in\mbox{int}(D^2)}\mbox{Area}(\Sigma_{r,\theta})\leq \sup_{(r,\theta)\in\mbox{int}(D^2)}\mathcal{W}(\Sigma_{r,\theta})=\mathcal{W}(\Sigma).  
\end{equation}
As such, when $e(\Sigma)\neq 0$, the bound \eqref{willmorebound} follows from Theorem \ref{conflowerbounds} as does the equality case.
\end{proof}

Finally, let us prove 
\begin{thm}
Let $\Sigma$ be a Klein bottle immersed in $\mathbb{S}^3$ containing the geodesic $C$ and invariant with respect to Schwarz reflection through $C$. Then
\begin{equation}\label{doublelower}
\mathcal{W}(\Sigma)\geq 2m_0
\end{equation}
with equality if and only if $\Sigma$ is a conformal dilate of an element in $\mathcal{M}^*$.
\end{thm}
\begin{proof}
By Corollary \ref{doublesit}, we get
\begin{equation}
\mathcal{V}_c(\Sigma)\geq 2\mbox{Area}(\Gamma_0).  
\end{equation}
By a result of Li-Yau (Lemma 1 in \cite{LY}) for a closed immersed surface $\Sigma$ in $\mathbb{S}^3$ there holds
\begin{equation}
 \mathcal{W}(\Sigma)\geq \mathcal{V}_c(\Sigma).
\end{equation}
The bound \eqref{doublelower} then follows from Corollary \ref{doublesit}, as does the equality case.
\end{proof}

Post-composing $\tilde{F}$ with a conformal map of $(A_b,g_{A_b})$, we can assume the renormalized $\tilde{F}$ induces the same marking as $j_b^{-1}$.


\section{Mollification}\label{appendix}

\subsubsection{Preliminaries}
Let $\Sigma$ be a surface with one boundary component and some fixed background Riemannian metric. Given a $W^{1,2}$ map $u:\Sigma\to\mathbb{S}^3$, the differential $du_x:T_x(\Sigma)\to T_{u(x)}\mathbb{S}^3$ is defined almost everywhere.  Denote by $J_u$ the Jacobian of the map $u$.  By the Cauchy-Schwartz inequality, $J_u\in L^1(\Sigma)$.  Thus to the map $u$ is associated a $2$-varifold that we denote $|u(\Sigma)|$ as follows.  For any $f\in C(G_2(\mathbb{S}^3))$ we define:
\begin{equation}\label{VarifoldToMap}
|u(\Sigma)|(f)= \int_\Sigma f(u(x), du_x(T_x\Sigma))J_u(x) d\mu_\Sigma.
\end{equation}
where here $\mu_\Sigma$ is the Riemannian volume density associated to the fixed metric.
Note that where $du_x(T_x\Sigma)$ fails to be a $2$-plane, $J_u$ vanishes. 
We denote the mass of the associated $2$-varifold $|u(\Sigma)|$ by:
\begin{equation}
\mbox{Area}(u)=\mbox{Area}(|u(\Sigma)|)=\int_\Sigma J_ud\mu_\Sigma.
\end{equation}



We have the following (cf. Proposition A.3 in \cite{CM1}):

\begin{lemma}\label{varcont}
If the sequence $\{u_i\}_{i=1}^\infty\in W^{1,2}(\Sigma,\mathbb{S}^3)$ converges to $u_\infty$ in the $W^{1,2}(\Sigma,\mathbb{S}^3)$ (strong) topology, then 
$$
|u_i(\Sigma)|\to |u_\infty(\Sigma)| 
$$
in the varifold weak topology.

In particular, given $\phi\in W^{1,2}\cap C^0(\Sigma,\mathbb {S}^3)$ spanning $C$ and $\eta>0$ there exists $r_1=r_1(\phi, \eta)>0$ so that if $||\phi-\phi'||_{W^{1,2}(M_1)}\leq r_1$ and $\phi'$ spans $C$ then 
\begin{equation}
\mathbf{F}(|\phi(\Sigma),|\phi'(\Sigma)|)\leq \eta. 
\end{equation}
\end{lemma}
\begin{proof}  Fix a $f\in C(G_2(\mathbb{S}^3))$, which is automatically uniformly continuous.
We note that the nature of the convergence of $u_i$ implies $J_{u_i}\to J_u$ in $L^1(\Sigma)$. Moreover, up to passing to a subsequence we may assume $u_i$ and $du_i$ converge a.e. to $u_\infty$ and $du_{\infty}$ respectively. Now for given $\epsilon>0$, let $\delta>0$ be chosen so that $||f||_\infty  \delta<\epsilon$.    By Egorov's theorem applied to the measure $J_{u_\infty} d\mu_\Sigma$, there exists a set $A\subset \Sigma$ so that $u_i\to u_\infty$ and $du_i\to du_\infty$ uniformly on $\Sigma\setminus A$ and $J_{u_\infty}\mu_\Sigma(A)<\delta$ -- here $\mu_\Sigma$ is the Riemannian volume density associated to the fixed background metric.  It follows that for $i$ sufficiently large, 
$$|\left| u_i(\Sigma)|(f)-|u_\infty(\Sigma)\right|(f)| <3\epsilon.$$
The first claim follows from this. 
\end{proof}

\subsubsection{Mollification}
Let $\rho\in C_c^\infty(\mathbb{R}^2,\mathbb{R})$ be rotationally symmetric with $\mbox{supp}(\rho)\subset B_1(0)\subset\mathbb{R}^2\cong\mathbb{C}$, and satisfying $\int_{D^2}\rho d\mu=1$. Set $\rho_\varepsilon(x):=\frac{1}{\varepsilon^2}\rho(x/\varepsilon)$ for each $\varepsilon>0$ so that $\mbox{supp}(\rho_\varepsilon)\subset B_\varepsilon(0)$.  

For $\phi\in W^{1,2}\cap C^0(M_1,\mathbb{S}^3)$ spanning $C$ let us denote by $\tilde{\phi}\in W^{1,2}\cap C^0(K_1,\mathbb{S}^3)$ the map obtained by Schwarz reflecting about $C$:
\begin{equation}\label{defdouble}
\tilde{\phi}(x)=
\begin{cases}
\phi(x), & x\in M_1,\\[0.4em]
R\bigl(\phi(\tau_1((x)\bigr), & x\in K_1\setminus M_1.
\end{cases}
\end{equation}
Note that for $x\in \mbox{Fix}(\tau_1)=\partial M_1$ we have 
\begin{equation}
R(\phi(\tau_1(x)))=R(\phi(x))=\phi(x), 
\end{equation}
from which the continuity over $\partial M_1$ follows.  
From the definition \eqref{defdouble} we have
\begin{equation}\label{theycommute}
\tilde{\phi}\circ\tau_1(x)=R\circ \tilde{\phi}(x)\mbox{ for } x\in K_1.
\end{equation}

Extend $\tilde{\phi}$ to $\mathbb{C}$ (not relabelled) by pulling back via the covering map $\pi:\mathbb{C}\to K_1$.
For $\varepsilon>0$, let us set 
\begin{equation}\label{convdef}
\tilde{\phi}_\varepsilon = \rho_\varepsilon\ast\tilde{\phi},
\end{equation}
where the convolution is conducted component-wise in $\mathbb{R}^4$.  

By the rotational symmetry of $\rho_\varepsilon$, convolution commutes with any isometry $I$ of $\mathbb{C}$ in the deck group of the covering $\pi$ of $K_1$:
\begin{equation}
I \circ \tilde{\phi}_\varepsilon=I \circ (\rho_\varepsilon\ast\tilde{\phi})=\rho_\varepsilon\ast\tilde{\phi}(I(\cdot))=\rho_\varepsilon\ast\tilde{\phi}=\tilde{\phi}_\varepsilon.
\end{equation}
Thus $\tilde{\phi}_\varepsilon$ descends to an $\mathbb{R}^4$-valued function $W^{1,2}\cap C^0(K_1,\mathbb{R}^4)$.

Let us also define the projection map
\begin{equation}
\Pi:\mathbb{R}^4\setminus \{\mathbf{0}\}\to\mathbb{R}^4,
\end{equation}
with range $\mathbb{S}^3$ given by \begin{equation}\Pi(x) = \frac{x}{|x|}.\end{equation} 

Finally, for any $\varepsilon>0$ (small enough, to be specified later) we denote the \emph{$\varepsilon$-mollification $M_\varepsilon(\phi)\in C^\infty(M_1,\mathbb{S}^3)$ of $\phi$} by
\begin{equation}
M_\varepsilon(\phi)=(\Pi\circ \tilde{\phi}_\varepsilon)|_{M_1}.
\end{equation}
Note again that since $\rho_\varepsilon$ is rotationally symmetric, the mollification operation commutes with the isometry $\tau_1$ of $K_1$ and we have from \eqref{theycommute}
\begin{equation}\label{stillcommute}
\tilde{\phi_\varepsilon}\circ\tau_1(x)=R\circ \tilde{\phi}_\varepsilon(x)\mbox{ for } x\in K_1.
\end{equation}
If $x\in\partial M_1$, there holds $x=\tau_1(x)$ and thus from \eqref{stillcommute} we obtain
\begin{equation}\label{boundaryfixed}
\tilde{\phi}_\varepsilon(x) = \tilde{\phi}_\varepsilon(\tau_1(x))= R\circ\tilde{\phi}_\varepsilon(x).
\end{equation}
Since the fixed set $\mbox{Fix}_{\mathbb{R}^4}(R)$ of $R$ is the two-dimensional subspace of $\mathbb{R}^4$ containing $C$, \eqref{boundaryfixed} gives that $\tilde{\phi}_\varepsilon(x)$ lies in this subspace when $x\in\partial M_1$.  Thus since $\Pi(\mbox{Fix}_{\mathbb{R}^4}(R)\setminus \{\mathbf{0}\})\subset C$ we obtain for all $\varepsilon$ small enough \begin{equation}M_\varepsilon(\phi)(x)\in C\mbox{ for } x\in \partial M_1.\end{equation} 

Next we collect properties of the mollification operation:

\begin{lemma}\label{mollprop} Suppose $\phi\in W^{1,2}\cap C^0(M_1,\mathbb{S}^3)$ spans $C$. Then there exists $\varepsilon_0=\varepsilon_0(\phi)>0$ and $C_\phi>1$ so that the following holds. For  $0<\varepsilon<\varepsilon_0(\phi)$:
\begin{enumerate}[(a)]
\item $M_\varepsilon(\phi)\in C^\infty(M_1,\mathbb{S}^3)$.
\item $M_\varepsilon(\phi)|_{\partial M_1}$ is a degree $1$ map to $C$.
\item $||M_\varepsilon(\phi)-\phi||_{W^{1,2}\cap C^0(M_1)}\leq \nu(\phi,\varepsilon)$ with $\nu(\phi,\varepsilon)$ decreasing to $0$ as $\varepsilon\to 0$.
\item If $\phi'\in W^{1,2}\cap C^0(M_1,\mathbb{S}^3)$ satisfies $||\phi-\phi'||_{W^{1,2}\cap C^0(M_1)}\leq \frac{1}{4}$ and $\phi'$ spans $C$ then there holds \begin{equation}||M_\varepsilon(\phi)-M_\varepsilon(\phi')||_{W^{1,2}\cap C^0(M_1)}\leq C_\phi ||\phi-\phi'||_{W^{1,2}\cap C^0(M_1)},\end{equation} and $M_\varepsilon(\phi')$ spans $C$ for any $0\leq \varepsilon\leq \varepsilon_0(\phi)$. 
\end{enumerate}

\end{lemma}
\begin{proof}
By standard properties of convolutions, $\tilde{\phi}_\varepsilon\to \tilde{\phi}$ in $C^0(K_1,\mathbb{R}^4)$ as $\varepsilon\to 0$.  Since $\tilde{\phi}$ has target $\mathbb{S}^3$, there exists $\varepsilon_0=\varepsilon(\phi)>0$ so that 
\begin{equation}\label{above}
|\tilde{\phi}_\varepsilon|\geq \frac{3}{4}, 
\end{equation}
for all $0<\varepsilon<\varepsilon_0$.  Note that on the region $\Omega_{1/2}=\{|x|\geq 1/2\}\subset\mathbb{R}^4$, the projection map $\Pi$ has uniformly bounded derivative:
\begin{equation}\label{proj}
|d_x\Pi(v)|\leq 2|v|\mbox{ for all } x\in\Omega_{1/2}\mbox{ and } v\in\mathbb{R}^4.
\end{equation}
We also have for some $C>0$:
\begin{equation}\label{proj2}
|d_x\Pi(v)-d_y\Pi(v)|\leq C|x-y|\cdot |v|\mbox{ for } x,y\in \Omega_{1/2}\mbox{ and } v\in \mathbb{R}^4.
\end{equation}
Since $\tilde{\phi}$ is $\mathbb{S}^3$-valued we get
\begin{equation}
||\Pi(\tilde{\phi}_\varepsilon)-\tilde{\phi}||_{C^0(K_1)}=||\Pi(\tilde{\phi}_\varepsilon)-\Pi(\tilde{\phi})||_{C^0(K_1)}.
\end{equation}
Thus using \eqref{proj} and the fundamental theorem of calculus we obtain
\begin{equation}\label{c0}
||\Pi(\tilde{\phi}_\varepsilon)-\tilde{\phi}||_{C^0(K_1)}\leq C||\tilde{\phi}_\varepsilon-\tilde{\phi}||_{C^0(K_1)}.
\end{equation}
Similarly, 
\begin{equation}\label{l2}
||\Pi(\tilde{\phi}_\varepsilon)-\tilde{\phi}||_{L^2(K_1)}\leq C||\tilde{\phi}_\varepsilon-\tilde{\phi}||_{L^2(K_1)}.
\end{equation}
By the chain rule we obtain for the ($\mathbb{R}^4$-valued) gradient
\begin{equation}
\nabla (\Pi\circ\tilde{\phi_\varepsilon})(x)= d_{\tilde{\phi}_\varepsilon(x)}\Pi\circ \nabla\tilde{\phi}_\varepsilon, 
\end{equation}Thus we can write
\begin{equation}\label{gradient}
\begin{aligned}
\nabla(\Pi\circ\widetilde{\phi}_\varepsilon)
-
\nabla\widetilde{\phi}
&=
d_{\widetilde{\phi}_\varepsilon}\Pi\circ \nabla\widetilde{\phi}_\varepsilon
-
d_{\widetilde{\phi}}\Pi\circ\nabla\widetilde{\phi} \\
&=
d_{\widetilde{\phi}_\varepsilon}\Pi\circ
\bigl(
\nabla\widetilde{\phi}_\varepsilon
-
\nabla\widetilde{\phi}
\bigr) \\
&\quad+
\bigl(
d_{\widetilde{\phi}_\varepsilon}\Pi
-
d_{\widetilde{\phi}}\Pi
\bigr)\circ 
\nabla\widetilde{\phi}.
\end{aligned}
\end{equation}
We obtain from \eqref{gradient} using \eqref{proj} and \eqref{proj2}:
\begin{equation}\label{grad}
||\nabla(\Pi(\tilde{\phi}_\varepsilon))-\nabla \tilde{\phi}||_{L^2(K_1)}\leq C||\nabla\tilde{\phi}_\varepsilon-\nabla\tilde{\phi}||_{L^2(K_1)}+C||\tilde{\phi}_\varepsilon-\tilde{\phi}||_{C^0(K_1)}\cdot||\nabla\tilde{\phi}||_{L^2(K_1)}, 
\end{equation}
where $C$ is independent of $\phi$ and $\varepsilon$.  Taken together, equations \eqref{c0}, \eqref{l2} and \eqref{grad} imply
\begin{equation}\label{boundhere}
||M_\varepsilon(\phi)-\phi||_{W^{1,2}\cap C^0(M_1)}\leq C_\phi||\tilde{\phi}_\varepsilon-\tilde{\phi}||_{W^{1,2}\cap C^0(K_1)}.
\end{equation}
The equation \eqref{boundhere} gives item (c) with 
\begin{equation}
\nu(\phi, \varepsilon):=C_\phi\cdot\sup_{0<s\leq\varepsilon} ||\tilde{\phi}_\varepsilon-\tilde{\phi}||_{W^{1,2}\cap C^0(K_1)}, 
\end{equation}
which is manifestly non-increasing as $\varepsilon$ is decreased.  Note that \eqref{c0} also implies item (b) (shrinking $\varepsilon_0$ if necessary) as it gives that $M_\varepsilon(\phi)\to \phi$ in $C^0(\partial M_1)$ as $\varepsilon\to 0$ and $\phi|_{\partial M_1}$ is a degree $1$ map to $C$ by assumption and the degree is unchanged under uniform limits. 

To see item (d), observe that if $||\phi-\phi'||_{C^0(M_1)}\leq\frac{1}{4}$, then also
\begin{equation}
\sup_{x\in K_1}|\tilde{\phi}_\varepsilon(x)-\tilde{\phi'_\varepsilon}(x)|\leq \frac{1}{4}\cdot||\rho_\varepsilon||_{L^1(K_1)}=\frac{1}{4}
\end{equation}
By \eqref{above} and the triangle inequality we obtain from this that
\begin{equation}\label{alsolower}
|\tilde{\phi'_\varepsilon}(x)|\geq |\tilde{\phi}_\varepsilon(x)|-|\tilde{\phi'_\varepsilon}(x)-\tilde{\phi_\varepsilon}(x)|\geq \frac{3}{4}-\frac{1}{4}=\frac{1}{2}.
\end{equation}
Analogously to \eqref{grad} we obtain (as by \eqref{alsolower} $\tilde{\phi'_\varepsilon}$ also maps to $\Omega_{1/2}$):
\begin{equation}
||\nabla(\Pi(\tilde{\phi}_\varepsilon))-\nabla(\Pi(\tilde{\phi'_\varepsilon}))||_{L^2(K_1)}\leq C||\nabla\tilde{\phi}_\varepsilon-\nabla\tilde{\phi'_\varepsilon}||_{L^2(K_1)}+C||\tilde{\phi}_\varepsilon-\tilde{\phi'_\varepsilon}||_{C^0(K_1)}\cdot||\nabla\tilde{\phi_\varepsilon}||_{L^2(K_1)}, 
\end{equation}
for $C>0$ independent of $\phi$ and $\varepsilon$.  Note also that (recalling the definition \eqref{convdef}):
\begin{equation}
||\tilde{\phi_\varepsilon}-\tilde{\phi'_\varepsilon}||_{C^0(K_1)}\leq ||\phi'-\phi||_{C^0(K_1)}, 
\end{equation}
and by Young's inequality for convolutions
\begin{equation}
||\nabla\tilde{\phi_\varepsilon}-\nabla\tilde{\phi'_\varepsilon}||_{L^2(K_1)}\leq ||\rho_\varepsilon||_{L^1(K_1)}\cdot||\nabla \tilde{\phi'}-\nabla \tilde{\phi}||_{L^2(K_1)} =||\nabla \tilde{\phi'}-\nabla \tilde{\phi}||_{L^2(K_1)} .
\end{equation}
Similarly by Young's inequality, 
\begin{equation}
||\tilde{\phi_\varepsilon}-\tilde{\phi'_\varepsilon}||_{L^2(K_1)}\leq ||\tilde{\phi'}- \tilde{\phi}||_{L^2(K_1)}. 
\end{equation}
Since 
\begin{equation}
||\tilde{\phi}-\tilde{\phi'}||_{W^{1,2}\cap C^0(K_1)}\leq \sqrt{2}||\phi-\phi'||_{W^{1,2}\cap C^0(M_1)},
\end{equation}
we get from the preceding inequalities
\begin{equation}\label{changeof}
  ||M_\varepsilon(\phi)-M_\varepsilon(\phi')||_{W^{1,2}\cap C^0(M_1)}\leq C_\phi||\phi-\phi'||_{W^{1,2}\cap C^0(M_1)}, 
\end{equation}
for all $\varepsilon\in (0,\varepsilon_0)$, where (increasing it if necessary) $C_\phi>1$ depends on $\phi$ but not $\varepsilon$.
This completes the proof of item (d).
\end{proof}

\section{Problems}\label{problems}
We collect some problems and questions related to the results of this paper.  

The following is natural (though a new degree argument related to, but distinct from, Proposition \ref{nontrivial} would likely be needed):   

\begin{question}
Can one use the mean curvature flow in $\mathbb{S}^3$ to prove Theorems \ref{mainintro} and \ref{ConfAreaLBThm}?  
\end{question}

See the work of B. White \cite{White3pi} for study of mean curvature flow with boundary and the forcing of suitable singularities.  

Since the conformal family of a non-orientqble surface is non-trivial for all positive Euler numbers (cf. Proposition \ref{nontrivial} and Remark \ref{sameforgenus}), developing a boundary version of the min-max theory for higher genus surfaces obtained by Zhou \cite{Zhougenera}, one should be able to show the following:
\begin{question}\label{betterversion}
The area of a non-orientable minimal surface in $\mathbb{S}^3$ spanning $C$ with non-zero Euler number is bounded from below by that of an embedded non-orientable minimal surface with index $2$ and boundary $C$. 
\end{question}
Note that the min-max surface one obtains as a lower bound in Question \ref{betterversion} could have Euler number zero.  We expect, on the other hand:
\begin{question}
The only minimal surfaces in $\mathbb{S}^3$ bounded by $C$ with Euler number zero are the hemispheres. 
\end{question}

Hardt-Simon \cite{HS} ruled out the existence of orientable and embedded minimal surfaces with boundary $C$ other than the hemispheres. The following however is still open:
\begin{question}
Does $C$ bound any \emph{immersed} orientable minimal surface aside from the hemispheres.
\end{question}
We may ask the following analog of Theorem \ref{mainintro} for higher genera surfaces (analogous to a long-standing conjecture of Kusner in the closed setting):
\begin{question}
For $m$ an even integer at least $3$, show that the least area non-orientable embedded minimal surface of genus $m$ bounded by $C$ is the surface $\tilde{\xi}_{1,m-1}$ (constructed in \cite{BernKetExistence}).
\end{question}
For $0<\lambda<\frac{\pi}{2}$, restricting the canonical family to \begin{equation}H_\lambda(C):=\mathbb{S}^3\setminus T_\lambda(C),\end{equation} gives a $2$-parameter sweepout of the genus one handlebody $H_\lambda(C)$ associated to the connected component of $\mathcal{Z}_2(H_\lambda(C),\partial H_\lambda(C); \mathbb{Z}_2)$, the space of mod $2$ relative cycles, that contains a meridian disk.  

When $\lambda>\frac{\pi}{4}$, $H_\lambda(C)$ is mean convex and the study of the free boundary $p$-widths (\cite{MNW}) is a well-posed problem (\cite{LiZhouFree}, \cite{GuangLiWangZhouFree}).  Thus one can ask:
\begin{question}
Determine the (low) free boundary $p$-widths of the handlebody $H_\lambda(C)$ for $\frac{\pi}{4}\leq\lambda<\frac{\pi}{2}$. 
\end{question}
With regard to the Lawson band:
\begin{question}
Do restrictions of the Lawson M\"obius band $\bar{\tau}_{1,2}$ to $H_\lambda(C)$ solve a variational problem in the mean convex handlebodies $H_\lambda(C)$?  
\end{question}
One can also consider free boundary min-max surfaces in the \emph{non} mean-convex domains $H_\lambda(C)$ (where one obtains smooth free boundary minimal surfaces which may not be properly embedded (\cite{LiZhouFree})) and study their limit as $\lambda\to 0$. 

In a related, perhaps simpler, geometry, the following problem is natural:
\begin{question}
Classify minimal M\"obius bands and annuli in $\mathbb{D}^2\times\mathbb{S}^1$.
\end{question}
Note that the quotients of restrictions of suitable dilations of the helicoid in $\mathbb{R}^3$ give examples of such M\"obius bands and annuli. For constructions of minimal M\"obius bands in other rotationally symmetric solid tori, see \cite{Schulz}.  

Closed minimal surfaces in $\mathbb{S}^3$ and free boundary minimal surfaces in the Euclidean three-ball share deep connections to spectral geometry and extremal eigenvalue problems (\cite{Berger1973, Nadirashvili1996, ElSoufiIlias2000,FSSteklov}).  
\begin{question}
Is there a relationship between the theory of minimal surfaces spanning $C$ and spectral geometry?
\end{question}
We note that for free boundary minimal surfaces in the handlebody $H_\lambda(C)$ the coordinate functions are Steklov eigenfunctions for appropriate Steklov eigenvalues.

On a related topic, note that the Lawson band $\bar{\tau}_{1,2}$ is foliated or ruled by geodesic meridians.  While the minimizers for various systolic functional on M\"obius bands are known (cf. \cite{SabourauYassine2016} and the references therein), one can ask:
\begin{question}
Does the Lawson band $\bar{\tau}_{1,2}$ have any criticality property with respect to   systolic functionals on M\"obius bands?
\end{question}

Finally, we have the following question (attributed in the lecture notes \cite{LectureNotesOtis} to B. White):
\begin{conj}
The cone over $\bar{\tau}_{1,2}$ in $\mathbb{R}^4$ is area-minimizing.  
\end{conj}
During the final preparation of this manuscript a solution to this conjecture appeared \cite{guaracoWhitesConeMobius2026}.  One may then pose a special case of Conjecture \ref{secondArea}:
\begin{conj}
The cone over $\bar{\tau}_{1,2}$ in $\mathbb{R}^4$ is the lowest density boundary singularity in $\mathbb{R}^4$.
\end{conj}
 \printbibliography
\end{document}